\documentclass[11pt,reqno]{amsart}

\usepackage{amsmath}
\usepackage{amssymb}
\usepackage{amsfonts}
\usepackage[foot]{amsaddr}

\usepackage[left=30mm,right=30mm,top=30mm,bottom=30mm,marginparwidth=70pt]{geometry}
\usepackage{mathrsfs}
\usepackage{hyperref}
\usepackage{cite}

\usepackage[sc]{mathpazo}
\usepackage{graphics}
\usepackage{graphicx}
\usepackage{xcolor}
\usepackage{empheq}
\usepackage{comment}
\usepackage{marginnote}

\newcommand{\Z}{\mathbb{Z}}
\newcommand{\N}{\mathbb{N}}
\newcommand{\R}{\mathbb{R}}
\newcommand{\C}{\mathbb{C}}
\newcommand{\I}{\mathbb{I}}

\newcommand{\B}{\mathscr{B}}
\newcommand{\D}{\mathbf{D}}

\renewcommand{\H}{\mathscr{H}}
\newcommand{\Res}{{\mathscr{R}}}
\newcommand{\J}{\mathscr{J}}

\renewcommand{\P}{\mathscr{P}}

\renewcommand{\rho}{\varrho}

\newcommand{\E}{\mathscr{E}}

\DeclareMathOperator{\dist}{dist}
\DeclareMathOperator{\diam}{diam}
\DeclareMathOperator{\area}{area}
\DeclareMathOperator{\vol}{vol}

\newcommand{\la}{\langle}
\newcommand{\ra}{\rangle}

\newcommand{\eps}{\varepsilon}
\newcommand{\e}{_{\varepsilon}}
\newcommand{\ed}{_{\eps,\delta}}
\newcommand{\ke}{_{k,\varepsilon}}
\newcommand{\ie}{_{i,\varepsilon}}
\newcommand{\je}{_{j,\varepsilon}}

\newcommand{\al}{\alpha}

\renewcommand{\d}{\,\mathrm{d}}

\newcommand{\ds}{\displaystyle}

\newcommand{\Id}{\mathrm{I}}

\newcommand{\suml}{\sum\limits}

\newcommand{\wt}{\widetilde}
\newcommand{\wh}{\widehat}

\newcommand{\ceq}{\coloneqq}
\newcommand{\restr}{\!\restriction}

\newcommand{\A}{\mathscr{A}}

\newcommand{\HS}{\mathscr{H}}

\theoremstyle{plain}
\newtheorem{theorem}{Theorem}[section]
\newtheorem*{theorem*}{Theorem}
\newtheorem{lemma}[theorem]{Lemma}
\newtheorem*{lemma*}{Lemma}

\theoremstyle{remark}
\newtheorem{remark}[theorem]{Remark}
\newtheorem*{remark*}{Remark}

\newtheorem*{example*}{Example}
\theoremstyle{definition}

\numberwithin{equation}{section}

\title
[Homogenization and operator estimates for Steklov problems in perforated domains]
{Homogenization and operator estimates for Steklov problems in perforated domains}

\author[Andrii Khrabustovskyi]{Andrii Khrabustovskyi\,$^1$}
\address{$^1$\,Department of Physics, Faculty of Science, University of Hradec Kr\'{a}lov\'{e}, Rokitansk\'eho 62, 50003 Hradec Kr\'alov\'e, Czech Republic}
\email{andrii.khrabustovskyi@uhk.cz}

\author[Jari Taskinen]{Jari Taskinen\,$^2$}
\address{$^2$\,Department of Mathematics and Statistics, University of Helsinki, P.O.Box 68, Pietari Kalmin katu 5, 00014 Helsinki, Finland}
\email{jari.taskinen@helsinki.fi}

\begin{document}

	\begin{abstract}
		Let  the set $\Omega_\varepsilon$ be obtained from the bounded domain $\Omega$ by removing a family of $\varepsilon$-periodically distributed  identical balls. In $\Omega_\varepsilon$ one considers the   Steklov spectral problem. It is known from \texttt{[Girouard-Henrot-Lagacé, ARMA (2021)]} that, if the radii of the holes shrink at a critical rate such that the surface area of a single hole is comparable to the volume of a periodicity cell, then, in the limit $\varepsilon \to 0$, the Steklov spectrum converges to the spectrum of the problem $-\Delta u=\lambda Q u$ on
		$\Omega$ with some weight $Q>0$.
		In the present work, we extend this result by proving, under fairly general assumptions on the locations and shapes of the holes, convergence of the associated resolvent operators in the operator norm topology, together with quantitative estimates for the Hausdorff distance between the spectra. The underlying domain $\Omega$ is not assumed to be bounded.
	\end{abstract}
	

	\maketitle
	\allowdisplaybreaks	
	\thispagestyle{empty}
	
	\begin{center}
		\emph{Dedicated to the memory of Volodymyr O. Marchenko (1922--2026), a great mathematician}
	\end{center}
	
	\section{Introduction\label{sec1}}
	
	Homogenization in perforated domains is a classical topic in homogenization theory that traces back to the pioneering work by V.\,O.~Marchenko and E.\,Ya.~Khruslov \cite{MK64}.
	This class of problems addresses the asymptotic behaviour of solutions to boundary value problems posed in domains containing a large number of pairwise disjoint holes, in the limit where the number of holes (per unit volume) tends to infinity and their radii tend to zero. 
	
	Most of the existing literature in this area is devoted to Dirichlet, Neumann and Robin  problems (see, for instance,   \cite{CM82,CS79,RT75,Kh79,Ka85,Br88,CD88,MK64} and the monograph \cite{MK06}). In contrast, the homogenization of Steklov spectral problems 
	\begin{gather}\label{Steklov}
		\begin{array}{cl}
			\Delta u\e = 0&\text{ in }\Omega\e,\\[1mm]
			\ds{\partial u\e\over \partial\nu}=\lambda\e u\e&\text{ on }\partial D\ie
		\end{array}
	\end{gather}
	has received comparatively little attention. Hereinafter, $\Omega\e$ is a perforated domain  depending on a small parameter $\eps>0$, which characterizes the typical size of the microstructure; it is obtained 
	by making a lot of holes $D\ie$  in some fixed domain $\Omega$. The problem \eqref{Steklov} must be complemented by suitable boundary conditions on the exterior boundary $\partial\Omega$, but since these conditions do not play an essential role in the homogenization process, we omit them in this introductory section for the sake of brevity. 
	
	The relative lack of attention to Steklov problems is apparently due to the fact that such problems are, in a certain sense, atypical: the spectral parameter $\lambda\e$ enters through the boundary condition rather than through the governing differential operator in the bulk, which leads to limit problems that differ substantially from those arising in the Dirichlet, Neumann and Robin settings and the analysis   require different   techniques.
	
	To the best of our knowledge, the first homogenization result for the Steklov problem \eqref{Steklov} (with a general second-order elliptic operator replacing the Laplacian) was established by M.~Vanninathan in \cite{Va81}, where the problem was studied in a domain perforated by identical holes whose radii are of order $\eps$ and which are distributed 
	$\eps$-periodically. The author proved the convergence $\eps^{-1}\lambda\ke\to \lambda_k$ of the re-scaled 
	$k$th eigenvalue to the 
	$k$th eigenvalue of a certain homogenized elliptic operator with constant coefficients defined on the   unperturbed domain.
	Later on, this result was extended in \cite{Me95}, where 
	the complete asymptotic expansions for the
	eigenvalues and the corresponding eigenfunctions were constructed, and in 
	\cite{Pa95}, where the estimates on  the rate of convergence were derived.
	The numerical aspects of the above problem were investigated in \cite{YMHCJ21}.
	Also, we mention the works \cite{CNP12,Du13}, which address the homogenization of the Steklov problem with a sign-changing weight in front of the spectral parameter, considered within the same periodic geometric setting as in \cite{Va81}.  
	
	Recently, this topic was revived by A.~Girouard, A.~Henrot, and J.~Lagacé in \cite{GHL21}, where the authors investigated connections between the Steklov and Neumann eigenvalues through a homogenization limit of the Steklov problem.
	More precisely, they considered problem \eqref{Steklov}, but with a different scaling of the holes compared to \cite{Va81}. Namely, the radii of the holes $d\e$ were taken to satisfy
	$$
	d\e^{n-1}\eps^{-n} \to \beta > 0,
	$$
	where $n$ denotes the space dimension and $\varepsilon$ is the period. In other words, the holes shrink at a rate such that the surface area of a  hole is comparable to the volume of a period cell.
	This scaling is also known in the homogenization of the Robin Laplacian, where, in the homogenization limit, an additional potential term appears (see \cite{Ka85} for  details).

	It was proved in \cite{GHL21} that the homogenization of problem \eqref{Steklov}, subject to Steklov boundary conditions on the exterior boundary $\partial\Omega$, leads to the following spectral problem on the unperturbed domain $\Omega$:
	\begin{gather}\label{Steklov:hom}
		\begin{array}{cl}
			-\Delta u  = \lambda Q u &\text{ in }\Omega .
		\end{array}
	\end{gather}
	with Steklov boundary conditions on $\partial\Omega$.
	Here  $Q=\beta \varkappa_n$ with $\varkappa_n$ denoting the area of a unit sphere in $\R^n$. 
	Note that the choice of boundary conditions on  $\partial\Omega$ plays no essential role at this stage. The reason the authors chose to impose Steklov boundary conditions on $\partial\Omega$ is that they subsequently performed an asymptotic analysis of problem \eqref{Steklov:hom} as $\beta \to \infty$; in this regime, suitably rescaled eigenvalues of \eqref{Steklov:hom} converge to the eigenvalues of the Neumann Laplacian in $\Omega$. 
	
In subsequent works \cite{GL21,GKL21}, A.~Girouard, M.~Karpukhin and J.~Lagacé extended the result of \cite{GHL21} to the case where the holes are distributed along (generally non-flat) manifolds. The resulting homogenized spectral equation retains the same form; however, $Q$ is no longer constant but becomes a function.
In fact, it was shown in \cite{GKL21} that, by an appropriate choice of the holes, one can generate an arbitrary weight  $Q$   (from appropriate $L^p$ or 
Orlicz spaces).

	Notably, the above homogenization result has been used in \cite{GHL21,GL21,GKL21} to establish several non-trivial results in geometric analysis. For instance, it was proven in \cite{GL21} that Kokarev's  upper bound $8\pi$ for the first nonzero normalized Steklov eigenvalue on orientable genus-zero surfaces is saturated.
	\smallskip
	
	The goal  of the present work is to extend the homogenization result obtained in \cite{GHL21}. We consider problem \eqref{Steklov} posed in a perforated domain $\Omega\e$. Following \cite{Va81} we impose Dirichlet boundary conditions on the exterior boundary; however, all our results remain valid (up to minor technical modifications) if, as in \cite{GHL21}, one instead prescribes Steklov boundary conditions on $\partial\Omega$.
	We impose fairly general assumptions on the distribution and shapes of the holes (subject to certain uniform bounds; see \eqref{assump:shape1}--\eqref{assump:shape4}). 
	Furthermore, we allow the domain $\Omega$ to be unbounded; the assumptions imposed on the holes ensure that the associated Steklov problem is well defined. The study of Steklov problems in unbounded domains is partly motivated by applications to the linear water-wave problem; see, e.g., 
	\cite{CNT18,NT23,CCNT22}.
	
	Our first main result, Theorem~\ref{th1}, concerns the convergence of the resolvent operators $\Res\e$ associated with problem \eqref{Steklov} to the resolvent operator $\Res$ corresponding to the homogenized problem \eqref{Steklov:hom} in a uniform operator topology. Moreover, we provide quantitative bounds for the rate of this convergence. Namely, we prove the following estimate:
	\begin{gather}\label{mainest:1}
		\|\Res\e\J\e -\J\e\Res\|_{H^1(\Omega)\to H^1(\Omega\e)} \leq C\delta\e.
	\end{gather}
	Here, $\J\e$ denotes the restriction from $\Omega$ to $\Omega\e\subset\Omega$, 
	$\|T\|_{X \to Y}$ denotes the operator norm of a bounded linear operator $T$ acting from the Banach space $X$ into the Banach space $Y$, 
	$\delta\e$ is an explicitly given error rate, and the constant $C$ depends   on $\Omega$ and the above-mentioned geometric bounds on the holes.
	
	In homogenization theory, estimates of the form \eqref{mainest:1} are commonly referred to as \emph{operator estimates}. They originate from the pioneering works \cite{BS04,BS06,Gr04,Z05a,Z05b,ZP05,Gr06}, which addressed periodic elliptic operators with rapidly oscillating coefficients.
	Over the last decade, such kind estimates have been established for various boundary value problems in perforated domains. The Dirichlet problem was studied in \cite{AP21,KP18}, the Neumann problem in \cite{Su18,AP21,ZP05}, and the Robin problem in \cite{KP22}.
	More recently, these results were substantially strengthened in \cite{BK22,Bo23a,Bo23b}, where sharp estimates for Dirichlet and (nonlinear) Robin holes were obtained under very general assumptions on the geometry of the perforations. We also mention the earlier work \cite{BCD16}, which concerns surface distributions of holes with either Dirichlet or Robin boundary conditions. To the best of our knowledge, operator estimates in the homogenization of Steklov problems in perforated domains have not yet been addressed, neither for holes of the same order as the period, as in \cite{Va81}, nor for small holes, as in \cite{GHL21}.
	\smallskip
	
	Our second result,  Theorem~\ref{th2}, concern the convergence of spectra $\sigma(\Res\e)$ of the operator $\Res\e$;
	note that the eigenvalues $\mu\e$ of $\Res\e$ and the 
	Steklov eigenvalues $\lambda\e$ are linked 
	via the relation $\mu\e=(\lambda\e+1)^{-1}$.
	We prove the following bound:
	\begin{gather}\label{mainest:2}
		{\dist_H}(\sigma(\Res\e),\sigma(\Res))\leq C\delta\e,
	\end{gather}
	where $\dist_H(\cdot,\cdot)$ stands for the Hausdorff distance.
	
	Note that for two normal bounded operator acting at the same Hilbert space $\HS$, one has the following bound  \cite{HN99}: 
	$$\dist_H(\sigma(\Res_1),\sigma(\Res_2))\leq \|\Res_1-\Res_2\|_{\HS\to\HS}.$$
	However, in our setting the operators $\Res\e$
	and $\Res$ act on different Hilbert spaces. Consequently, the bound \eqref{mainest:2} does not follow directly from the estimate \eqref{mainest:1}. Establishing \eqref{mainest:2} therefore requires additional analysis.		
	\smallskip
	
	The paper is organized as follows. In Section~\ref{sec2}, we formulate the problem and state the main results. Section~\ref{sec3} contains auxiliary estimates used in the analysis. Finally, Section~\ref{sec4} is devoted to the proofs of the main results.

	\section{Problem setting\label{sec2}}

	\subsection{Geometry}
	
	Let $\Omega$ be an open domain in $\R^n$,  $n\geq 2$.
	We assume that $\partial\Omega$ is uniformly regular of class $C^2$ (see, e.g.\cite[Definition~1]{Br61}). 	This requirement is merely technical, for example, it is needed   to enjoy the global $H^2$-regularity of the solution to the homogenized problem. Also, it implies the existence
	of a constant $C_\Omega>0$ such that the inner $C_\Omega$-neighborhood of $\partial\Omega$
	is contained in $\Omega$ and
	the  map 
	\begin{gather}\label{tau}
		(x, t) \mapsto x -  t \nu(x)
	\end{gather}
	is one-to-one on $\partial\Omega \times [0, T]$ for any $T\le C_\Omega$;
	here $\nu(x)$ stands for the unit
	outward-pointing normal vector at $x\in \partial\Omega$.

	Let $\eps>0$ be a small  parameter. To simplify the presentation, we assume that it takes
	discrete values, specifically $\eps=m^{-1}$ 
	for $m\in\N$ (so that $\eps\to 0$ corresponds to $m\to\infty$). This
	assumption allows us to avoid repeatedly stating `for sufficiently small $\eps$' in our arguments: for instance, if a sequence $a\e$ converges as $\eps=m^{-1}\to 0$, it is automatically bounded, 	whereas for a continuously varying $\eps$, this holds for small enough $\eps$.
	
Next, we introduce a family of domains, each of which will  accommodate a hole.
	For any fixed $\eps>0$, let $\{\Gamma\ie\}_{i\in\I\e}$ be a family open sets in $\R^n$, $n\ge 2$; here the index set $\I\e $ is either finite
	or countable.
	We assume that  
	\begin{gather} 
		\text{each }\Gamma\ie\text{ is  {a convex polytop}},
		\qquad
		\label{Gamma:cond:1}
		\Gamma\ie\cap\Gamma\je=\emptyset,\ i\not=j,\qquad
		\cup_{i\in\I\e }\overline{\Gamma\ie}\subset\overline{\Omega}.
	\end{gather} 
Let $r\ie$ and $\wt r\ie$ be  the inner and the outer radii of   $\Gamma\ie$, i.e., the radius of the largest  ball contained in $\Gamma\ie$ and the radius of the smallest ball containing it, respectively. 
We  denote 
\begin{gather}
\label{re:Te:def}
r\e\ceq \sup_{i\in\I\e}r\ie,\qquad 
T\e\ceq\Omega\setminus\overline{ \cup_{i\in\I\e}\Gamma\ie }.
\end{gather}

We impose the following assumptions on the sets $\Gamma\ie$:
\begin{gather} 
		\label{Gamma:cond:2}
		\exists \wt C>0\ \forall\eps>0\ \forall i\in \I\e :\quad 
		\wt r\ie\leq \wt C r\ie,
\\
 \label{assump:re}
 	r\e\to 0\text{ as }\eps\to 0,
 \end{gather}
 and, moreover, there is
 a sequence $(\rho\e)_{\eps>0}$ such that 
 \begin{gather}\label{assump:Te}
 	\rho\e\in (0,C_\Omega),\quad \lim\limits_{\eps\to 0}\rho\e= 0,\quad
 	T\e\subset   \{y= x - t \nu(x)\in\R^n:\  x\in\partial\Omega,\ t\in (0,\,\rho\e)\}.
 \end{gather} 
 
	Now, let for each $\eps>0$  we 
	are given with a family of  Lipschitz domains $\{D\ie,\ i\in\I\e\}$
	satisfying
	$\overline{D\ie}\subset\Gamma\ie$; we assume that 
	the complement of each $D\ie$ in $\R^n$ is a connected set.
	Further assumptions on $D\ie$ will be given in the next subsection.

	Finally,  we define the perforated domain $\Omega\e$ as follows (see  Figure~\ref{fig2}), 
	\begin{gather}\label{Omega:e}
		\Omega\e \ceq \Omega \setminus\left( \bigcup_{i \in \I\e}\overline{ D\ie}\right).
	\end{gather}   
	
	\begin{figure}[t]
	\begin{picture}(200,145)
		\includegraphics[width=0.6\linewidth]{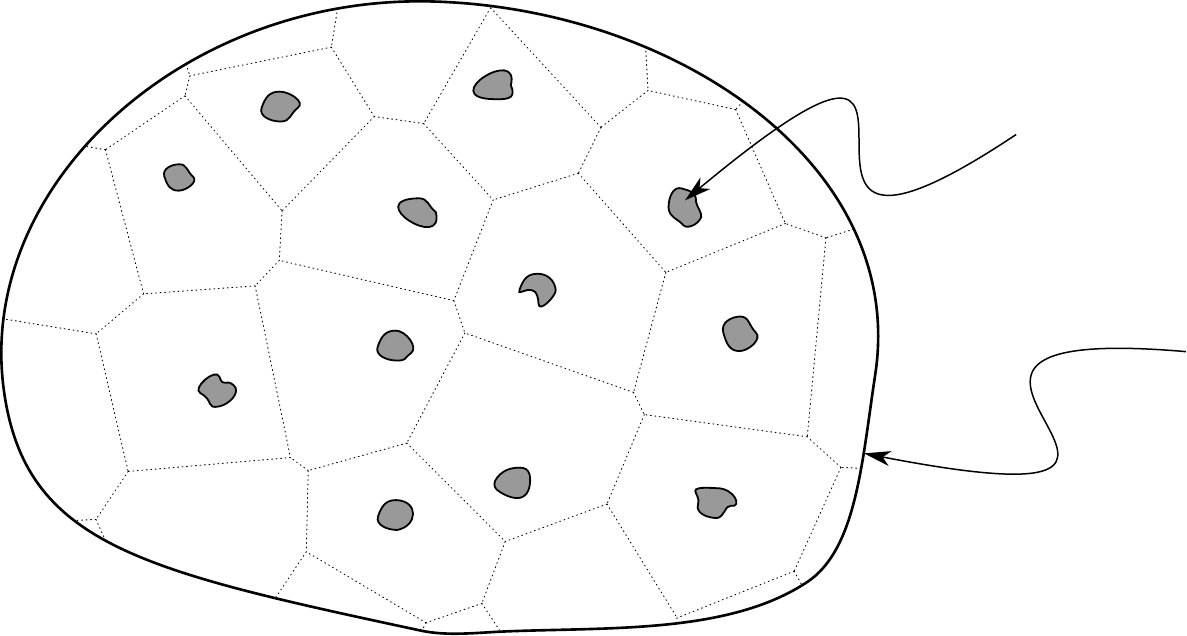}
		\put(-37,111){$D\ie$}
		\put(2,60){$\partial\Omega$}
	\end{picture}
	
	\caption{\label{fig2}The perforated domain $\Omega\e$ defined by \eqref{Omega:e}. The dotted lines correspond to the boundaries of the cells $\Gamma\ie$.}
\end{figure}

	We conclude this subsection with two examples of families $\{\Gamma\ie\}_i$ satisfying assumptions \eqref{Gamma:cond:1} and \eqref{Gamma:cond:2}. 

	 	The first examples  
	is given by a \emph{Voronoi diagram} generated by some countable set of points 
	$\mathcal{P}\e\subset\R^n$, see Figure~\ref{fig1}, left-hand picture. 
	The corresponding Voronoi cells (we denote them by $\Gamma\ie$ with $i\in\N$)  are pairwise disjoint, moreover, 
	 each cell is a convex polytope. 
	We also assume that the convex hull of $\mathcal{P}_\varepsilon$ coincides with $\mathbb{R}^n$ -- this assumption ensures that all $\Gamma\ie$   are bounded and that the union of their closures covers the entire   $\R^n$. Now, we set $\I\e=\{i\in\N:\ \Gamma\ie\subset\Omega\}$, then
	 the family $\{\Gamma\ie\}_{i\in\I\e}$ satisfies \eqref{Gamma:cond:1}.
	Next, we define the
	\emph{separation distance} $\mu\e$ and the \emph{covering radius} $\wt\mu\e$ 
	of the set $\mathcal{P}\e$
 as follows,
	\begin{align*}
		\mu\e&\ceq {\inf}_{x,y\in\mathcal{P}\e,\,x\not=y}\dist(x,y),\qquad
		\wt\mu\e\ceq {\sup}_{x\in\R^n}\dist(x,\mathcal{P}\e).
	\end{align*}
	It is easy to show that the inner radius $r\ie$  and the outer radius  $\wt r\ie$ of each $\Gamma\ie$  satisfy
	\begin{gather}	
		\label{mu:ineqs}
		\mu\e\leq 2r\ie,\quad \wt r\ie\leq\wt\mu\e.
	\end{gather}  
	Then, to ensure \eqref{Gamma:cond:2}, we assume that for each $\eps>0$ one has $\mu\e>0$, $\wt\mu\e<\infty$, and, moreover, $\mu\ceq \sup_{\eps>0}\wt\mu\e/\mu\e <\infty$.
	Under these assumptions,   \eqref{Gamma:cond:2}
	holds by  \eqref{mu:ineqs}, with $\wt C=2\mu$.
 
 The second example is obtained by taking the sets $\Gamma\ie$
 to be $n$-dimensional rectangles arranged so that \eqref{Gamma:cond:1} holds. In order to fulfill \eqref{Gamma:cond:2} 
	each rectangle is required to satisfy a uniform shape-regularity condition: the ratio between its longest and shortest side lengths is bounded uniformly by some absolute (i.e., independent of $\varepsilon$ and $i$) constant. Note that this condition does not preclude the rectangles from having vastly different sizes; see Figure~\ref{fig1} (right-hand picture).

	\begin{figure}[b]
		\includegraphics[width=0.4\linewidth]{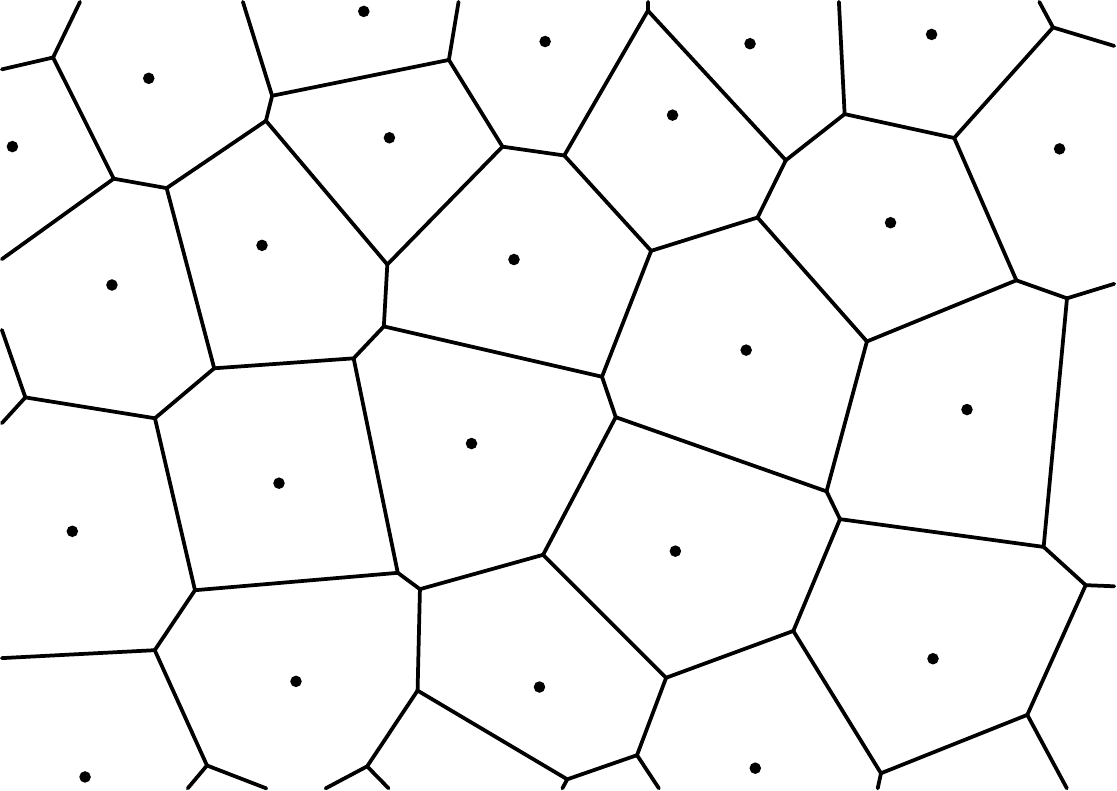}\qquad\qquad
		\includegraphics[width=0.4\linewidth]{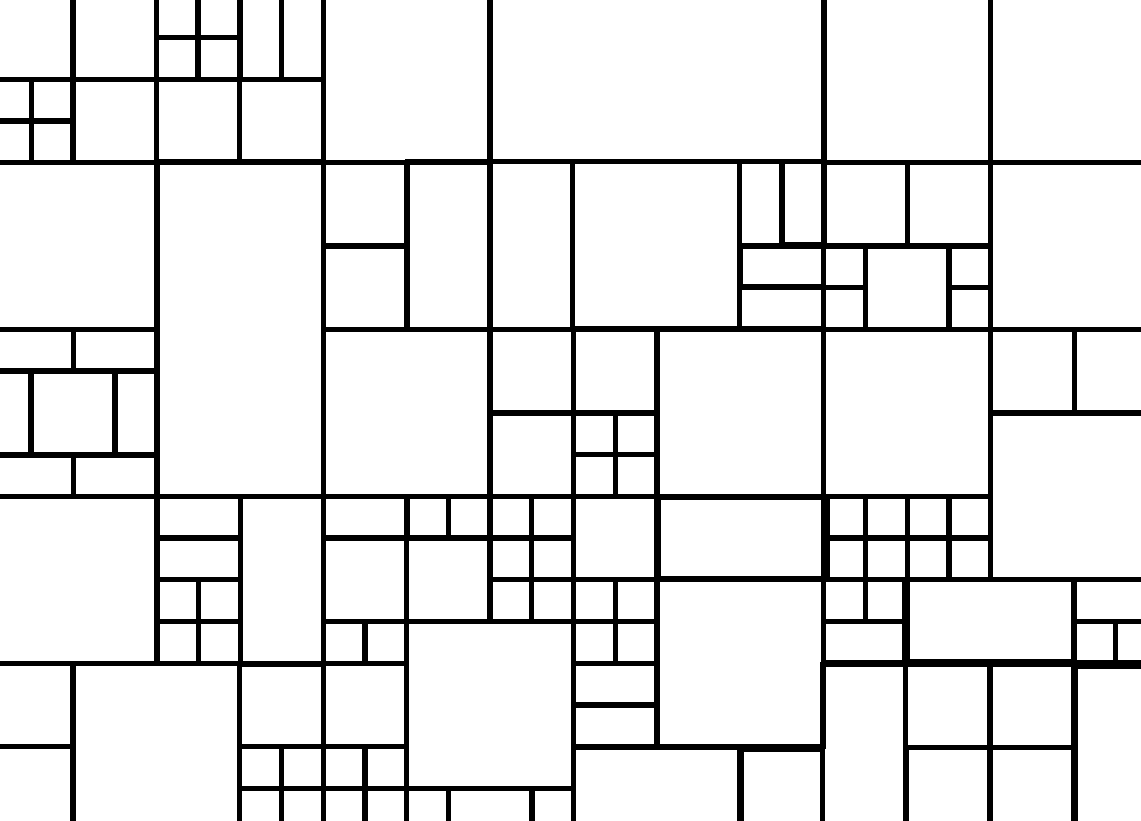}
		\caption{Examples of the families $\{\Gamma\ie\}$.
			\emph{Left:} a Voronoi tessellation generated by a countable set of points. 
			\emph{Right:} a rectangular tessellation.
			\label{fig1}}
	\end{figure}

	\subsection{Assumption on the `holes' $D\ie$}
	
	In what follows, we denote by $\B(r,z)$ the open ball in $\R^n$ of radius $r>0$ and center $z\in\R^n$.

	\subsubsection*{Assumptions on the shapes of $ D\ie $}\label{subsec:22}
	We formulate our assumption in terms of the upscaled sets
	$$
	\D \ie\ceq (d\ie )^{-1}(D\ie-x\ie),
	$$
	where 	by $d\ie$ and $x\ie$ we denote
	the radius and   center of the smallest ball containing   $D\ie$.
	Evidently, the smallest ball containing the set $\D\ie$ is $\B(1,0)$, i.e. the unit ball with the center at the origin.
	Note that by the convexity of $\Gamma\ie$, for each $ i\in\I\e$ the point $x\ie$ belongs to $\Gamma\ie$.
	We   denote 
	$$\mathbf{G}\ie\ceq \B(2,0)\setminus\overline{\D\ie},$$
	and  introduce the linear bounded operator  $$\mathbf{P}\ie:H^1(\mathbf{G}\ie)\to H^1(\B(2,0))$$ being defined by 
	$\mathbf{P}\ie u=\overline{u}$, where $\overline{u}$ stands for the harmonic extension on $\D\ie$,
	i.e., $\Delta \overline{u}=0$ on $\D\ie$, and $u=\overline{u}$ on $\mathbf{G}\ie$. 
	Then
	our assumptions read as follows:
	\begin{itemize} \setlength\itemsep{0.5em}
		\item The inner radii of the sets $\D \ie $ are bounded away from zero uniformly in $\eps $ and $i $:
		\begin{gather}\label{assump:shape1}
			\exists C_{\rm inn}>0\
			\forall \eps>0\ \forall i\in\I\e\
			\exists z \ie\in\R^n:\
			\B(C_{\rm inn},z\ie)\subset\D\ie.
		\end{gather}
		
		\item The trace operators $H^1(\D\ie)\to L^2(\partial \D\ie)$
		are bounded uniformly in $\eps$ and $i$:
		\begin{gather}\label{assump:shape2}
			\exists C_{\rm tr}>0\
			\forall \eps>0\ \forall i\in\I\e:\
			\|u\|_{L^2(\partial \D\ie)}\leq C_{\rm tr}\|u\|_{H^1(\D\ie)},\
			\forall u\in H^1(\D\ie).
		\end{gather}
		
		\item The smallest non-zero eigenvalues of the Neumann Laplacian on  $\mathbf{G}\ie$ are bounded away from zero uniformly in $\eps$ and $i$:
		\begin{gather}\label{assump:shape3}
			\exists C_{\rm N}>0\
			\forall \eps>0\ \forall i\in\I\e:\
			C_{\rm N}\le\Lambda_{\rm N}(\mathbf{G}\ie).
		\end{gather}
		Hereinafter, by $\Lambda_{\rm N}(D)$ we denote the smallest non-zero eigenvalue of the Neumann Laplacian on a domain $D\subset\R^n$.
		
		\item
		The operators $\mathbf{P}\ie$ are bounded uniformly in $\eps$ and $i$:
		\begin{align}
		\exists C_{\rm P}>0\ \forall\eps>0\ \forall i\in\I\e:\label{assump:shape4} \quad\| \mathbf{P}\ie u\|_{H^1(\B(2,0))}\leq C_{\rm P}\|  u\|_{H^1(\mathbf{G}\ie)},\ \forall u\in H^1(\mathbf{G}\ie).
		\end{align}

	\end{itemize}
	
	At the end of this section, we discuss the assumptions \eqref{assump:shape1}--\eqref{assump:shape4} in more detail. 
	
	\subsubsection*{Assumptions on the location and sizes of $D\ie$}

	We introduce the numbers
	\begin{gather*}
		{Q}\ie\ceq \frac{\area(\partial D\ie)}{ \vol(\Gamma\ie)},\ i\in\I\e.
	\end{gather*}
	Hereinafter, by $\vol(D)$ we denote the volume of
	a domain $D\subset\R^n$, by $\area(S)$ 
	we denote the area of a closed $(n-1)$-dimensional hypersurface $S$ in $\R^n$, i.e. $\area(S)=\int_S \d s_x$, where
	$\d s_x$ is the density of the surface measure on $S$.
	
	We also define  the   piecewise constant function ${Q}\e:\R^n\to\R$ by
	$$ 
	{Q}\e(x)=\sum_{i\in\I\e}{Q}\ie \mathbb{1}_{\Gamma\ie}(x),
	$$
	where $\mathbb{1}_{\Gamma\ie}$ is the indicator function of $\Gamma\ie$.
	Note that although  $Q\e$ as well as its limit $Q$ introduced further (see  \eqref{Q:basic}--\eqref{assump:main}) are supported in $\Omega$, we   prefer to regard them as functions defined on the whole space    $\R^n$. This convention will simplify the presentation.

For $k\in\Z^n$ we denote 
$$\square_k\ceq [0,1]^n+k.$$

We impose the following conditions:
	\begin{itemize}  
		
			\item The sets $D\ie$ are located on a certain `secure' distance from the cell boundary:
		\begin{gather}\label{assump:D1}
			\exists C_{\rm sec}>0\ \forall\eps>0\ \forall i\in\I\e:\	
			C_{\rm sec}r\ie\le \dist(x\ie,\partial\Gamma\ie) .	
		\end{gather}
		
		\item The sets $D\ie$ are not very large, namely 
		\begin{gather}\label{assump:D1+}
			\forall\eps>0\ \forall i\in\I\e:\	
			2d\ie\leq C_{\rm sec}r\ie  .	
		\end{gather}
	 
		\item  The numbers $Q\ie$ are  bounded away from zero uniformly in $\eps$ and $i$:
		\begin{gather}\label{assump:Q1}
			\exists C^-_Q>0\ \forall\eps>0\ \forall i\in\I\e: 	\quad C^-_Q\leq Q\ie.
		\end{gather}	
		
		\item The sequence of functions $Q\e$ converges as $\eps\to 0$ in the following sense: there exist    $Q:\R^n\to\R$ 
		and (for $n=2$) a  number $\sigma>0$ such that
		\begin{gather}\label{Q:basic}
			Q\in L^{2\mu}_{\rm loc}(\R^n),\quad
			Q=0\text{ in }\R^n\setminus\overline{\Omega},
			\\
			\label{assump:Q2}
			\exists C^+_Q>0\ 	\forall k\in\Z^n:\ 
			\|{Q}\|_{L^{2\mu}(\square_k )}\le C^+_Q,
		\end{gather}
		where 
		\begin{gather}\label{mu}
			\mu\ceq \begin{cases}n/2,&n\ge3,\\1+\sigma,&n=2,\end{cases}
		\end{gather}
		and, moreover, we have
			\begin{gather}\label{assump:main}
			\kappa\e\ceq\sup_{k\in\Z^n}\|{Q}\e-{Q}\|_{L^\mu(  \square_k )}\to 0.
		\end{gather}

	\end{itemize}
	Note that  \eqref{assump:Q1} combined with \eqref{assump:main} and \eqref{assump:Te} implies easily that
	\begin{gather}\label{rho>0}
		C_Q^-\le {Q}(x)\text{ for a.a. }x\in\Omega.
	\end{gather}

	\begin{remark}
		The stronger regularity assumption imposed on the limiting function $Q$ (namely, $Q \in L^{2\mu}_{\rm loc}(\R^n)$ rather than only $L^{\mu}_{\rm loc}(\R^n)$, cf.~\eqref{assump:main}) is in fact essential for our analysis. This assumption is used repeatedly throughout the proofs of the main results; in particular, it is needed to ensure the $H^2$-regularity of solutions to the homogenized equation.
	\end{remark}

\begin{remark}
	The presence of the cubes $\square_k$ in the assumptions \eqref{assump:Q2}, \eqref{assump:main}  is required only when the domain $\Omega$ is unbounded.
	If $\Omega$ is bounded, then assumptions  \eqref{Q:basic},   \eqref{assump:Q2},   \eqref{assump:main} may be replaced by
	\begin{gather*}
		\|Q\e-Q\|_{L^\mu(\Omega)}  \to 0,
		\ \text{where }
		Q \in L^{2\mu}(\Omega).
	\end{gather*}
	
\end{remark}

		\begin{remark} \label{rem:Q}
	It is easy to see that for any function $Q$ satisfying \eqref{Q:basic}, \eqref{assump:Q2}, \eqref{rho>0}, one can construct a family of holes $\{D\ie,\ i\in\I\e,\ \eps>0\}$ satisfying all the assumptions stated above.
	To prove this, we assume, for simplicity, that $\Omega$ is bounded and that $n\ge 3$; in particular, this implies that $Q\in L^{n}(\Omega)$. Given $Q$, we define the function $\wt Q\e:\R^n\to \R$ by
		\begin{gather}
			\label{wtQ:def}
			\wt Q\e(x)=\sum_{i\in\I\e}
			\wt Q\ie \mathbb{1}_{\Gamma\ie}(x),\text{ where }
			\wt Q\ie\ceq  {1\over \vol(\Gamma\ie)}\int_{\Gamma\ie}Q(x)\d x.
		\end{gather}
		Using \eqref{Gamma:cond:2}, \eqref{assump:re}, \eqref{assump:Te} one can easily show  that 
		\begin{gather}\label{wtQ}
		\forall p\in [1,n]:\quad \|\wt Q\e-Q\|_{L^p(\Omega)} \to 0\text{ as }\eps\to 0.
		\end{gather}
		We choose $D\ie$ to be balls of the radius 
		\begin{gather}\label{die}
		d\ie = \left(\frac{\wt Q\ie \vol(\Gamma\ie)}{\area(\partial\B(1,0))} \right)^{1/(n-1)}=(\area(\partial\B(1,0)))^{-1/(n-1)}\left(\int_{\Gamma\ie}Q(x)\d x\right)^{1/(n-1)}.
		\end{gather}
		With this choice of $d\ie$, the functions $Q\e$ and $\wt Q\e$ do coincide, and hence, by \eqref{wtQ}, the property \eqref{assump:main}
		 holds true.
		Furthermore, the re-scaled sets  
		$\D\ie$ are unit balls, whence the properties \eqref{assump:shape1}--\eqref{assump:shape4} are automatically satisfied.  
		Choosing the center $x\ie$ of the ball $D\ie$ at
		the center of the largest ball contained in $\Gamma\ie$,
		we conclude that   \eqref{assump:D1} holds
		with $C_{\rm sec}=1$ (in fact, one has the equality in \eqref{assump:D1}). It remains to verify the fulfillment of \eqref{assump:D1+}  (at least for sufficiently small $\eps$). Using  H\"older's inequality and \eqref{Gamma:cond:2}, we get 
		\begin{gather*}
		\int_{\Gamma\ie}Q(x)\d x\leq  
		(\vol(\Gamma\ie))^{(n-1)/n}\|Q\|_{L^{n}(\Gamma\ie)}
		\leq Cr\ie^{n-1} \|Q\|_{L^{n}(\Gamma\ie)},
		\end{gather*}
		whence we conclude the estimate
		\begin{gather*}
				d\ie\leq C r\ie \|Q\|_{L^{n}(\Gamma\ie)}^{1/(n-1)}.
		\end{gather*}
		Now, since $Q\in L^n (\Omega)$ and, due to \eqref{Gamma:cond:2} and \eqref{assump:re}, 
		$\sup_{i\in\I\e}\vol(\Gamma\ie)\to 0$ as $\eps\to 0$, 
		then $\sup_{i\in\I\e}\|Q\|_{L^{n}(\Gamma\ie)}\to 0 $ as $\eps\to0$.
		Hence, indeed $2d\ie\leq r\ie$ for small enough $\eps$.
		
	\end{remark}

	\subsection{The operator $\Res\e$}
	
	Next, we describe the   operator  $ \Res\e $, which will be the main object of our interest. In order to do this, we need to the following bounds: 
	\begin{multline}\label{assump:e}
		\forall\eps>0\ \exists C\e^\pm>0:\\ C\e^-\|u\|^2_{H^1(\Omega\e)}\leq 
		\|\nabla u\|^2_{L^2(\Omega\e)}+\sum_{i\in\I\e}\|u\|^2_{L^2(\partial D\ie)} \leq
		C\e^+\|u\|^2_{H^1(\Omega\e)},\ \forall u\in H^1(\Omega\e).
	\end{multline}	
	Evidently, when $\Omega$ is a bounded domain and the index set $\I\e$ is finite, both bounds in \eqref{assump:e} are satisfied.
	It turns out that if  $\Omega$ is an unbounded set and $\I\e$ is countable,
	then  \eqref{assump:e} also holds true -- owing to the above assumptions we impose on the geometry of $\Omega\e$. To simplify the presentation, we postpone the proof of this statement to Section~\ref{sec3}.   In fact, we will show that 
	the bounds \eqref{assump:e} hold with $\eps$-independent constants
	(cf.~\eqref{th1:norms}).
	
	We introduce the subspace $\HS\e$ of $H^1(\Omega\e)$,
	\begin{align*}
		\HS\e\ceq \{u\in H^1(\Omega\e):\ u=0\text{ on }\partial\Omega\},
	\end{align*}
	and define  the scalar product on $\HS\e$ by
	\begin{align}\label{He:scalar}
		(u,v)_{\HS\e}&\ceq (\nabla u,\nabla v)_{L^2(\Omega\e)}+\sum_{i\in\I\e}(u,v)_{L^2(\partial D\ie)}.
	\end{align}
	By virtue of \eqref{assump:e} the right-hand side of \eqref{He:scalar} remains finite, even if the index set $\I\e$ is infinite, and the norm $\|\cdot\|_{\HS\e}=(\cdot,\cdot)_{\HS\e}^{1/2}$ is equivalent to the standard Sobolev norm in $H^1(\Omega\e)$. Thus $\HS\e$ equipped with the scalar product $(u,v)_{\HS\e}$ is  a Hilbert space.

	For $f\in\HS\e$ we consider the conjugate-linear bounded functional  
	$F\e:\HS\e\to\C$ given by
	$$F\e[v]=\sum_{i\in\I\e}(f,v)_{L^2(\partial D\ie)}.$$
	By the Riesz representation theorem there exists a unique 
	$u\e \in \HS\e$ such that
	\begin{gather}\label{variat:eps}
		(u\e,v)_{\HS\e}=F\e[v],\quad \forall v\in \HS\e.	
	\end{gather}
	Then we define the operator $\Res\e:\HS\e\to\HS\e$ by $\Res\e f\ceq u\e$.  Its is easy to see that this operator is linear, bounded, non-negative and self-adjoint. 
	The function $u\e=\Res\e f$ is a weak solution to the problem
	\begin{gather*}
		\Delta u\e=0\text{ on }\Omega\e,\qquad
		{\partial u\e\over\partial\nu}+u\e=f\text{ on }\partial D\ie,\ i\in\I\e,\qquad
		u\e=0\text{ on }\partial\Omega,
	\end{gather*}
	where ${\partial\over\partial \nu}$ stands for the derivative along the
	exterior unit normal.
	Finally,  we observe that if the domain $\Omega\e$ is bounded, then  $\Res\e$ is a compact operator -- this follows   from the compactness of the trace operator
	$H^1(\Omega\e)\to L^2(\partial\Omega\e)$ and the bound 
	$$\forall f\in \HS\e:\quad \|\Res\e f\|_{\HS\e}\le \|f\|_{L^2(\cup_{i\in\I\e}\partial D\ie)}$$
	(the latter inequality is obtained  by testing \eqref{variat:eps} with $v=u\e$).

	The correspondence between the spectrum of $\Res\e$ and the 
	Steklov spectrum is as follows: 
	the number $\mu\e\not=0$ is an eigenvalue of $\Res\e$ with the eigenfunction $u\e$ iff
	$\lambda\e=\mu\e^{-1}-1$ is an eigenvalue of the Steklov problem 
	\begin{gather*}
		\Delta u\e=0\text{ on }\Omega\e,\qquad
		{\partial u\e\over\partial\nu}=\lambda\e u\e\text{ on }\partial D\ie,\ i\in\I\e,\qquad
		u\e=0\text{ on }\partial\Omega,
	\end{gather*}
	with the same eigenfunction $u\e$.
	
	Our goal is to describe the behaviour of $\Res\e$ and its spectrum as $\eps\to 0$.

	\subsection{Main results}
	
	To define the anticipated limiting operator $\Res$, 
	we introduce the Hilbert space $\HS=H^1_0(\Omega)$ with the weighted scalar product
	\begin{gather}\label{HS:product}
	(u,v)_{\HS}\ceq(\nabla u,\nabla v)_{L^2(\Omega)}+
	\int_{\Omega}{Q}(x)u(x)\overline{v(x)}\d x.
	\end{gather}
	Here ${Q}$ is a function standing 
	in the assumption \eqref{assump:main}; recall that this function is bounded away from zero (see~\eqref{rho>0})  and satisfies \eqref{Q:basic}--\eqref{assump:Q2}. 
	The assumptions we imposed on   $Q$ guarantees that the integral \eqref{HS:product}  is well-defined. Indeed, let $\nu$ be given by 
	\begin{gather}\label{nu}
		\nu\ceq 
		\begin{cases}
			2n/(n-2),&n\ge 3,\\
			2+2\sigma^{-1},&n=2
		\end{cases}
	\end{gather}
	(recall that $\sigma$ is the number appearing in \eqref{assump:main} as $n=2$, see \eqref{mu}). The Sobolev embedding theorem (see, e.g., \cite[Theorem~5.4]{Ad75}) implies that the space $H^1(\square_k)$  is continuously embedded to the space 
	$L^\nu(\square_k)$. Since  all $\square_k$ are congruent, the norm of this embedding is independent on $k$:
	\begin{gather}\label{Sobolev:nu}
	\exists C_\nu>0\ \forall k\in\Z^n\ \forall u\in H^1(\square_k):\quad \|u\|_{L^\nu(\square_k)}\leq C_\nu\|u\|_{H^1(\square_k)}.
	\end{gather}
	We denote by  $\E :H^1_0(\Omega)\to H^1(\R^n)$ the operator of extension by zero.
 Then, using H\"older's
	inequality (note that $\mu^{-1}+2\nu^{-1}=1$, where $\mu$ is defined in \eqref{mu}), we get
	\begin{align}\notag
		&\left|\int_{\Omega}{Q}(x)u(x)\overline{v(x)}\d x\right|=
		\left|\sum_{k\in\Z^n}\int_{\Omega\cap\,\square_k}{Q}(x)u(x)
		\overline{v(x)}\d x\right|
		\leq
		\sum_{k\in\Z^n}\int_{\square_k}\left|{Q}(x)\E u(x)
		\overline{\E v(x)} \right|\d x
		\\\notag
		&\quad\leq 
		\sum_{k\in\Z^n}
		\|Q\|_{L^\mu(\square_k)}
		\|\E u\|_{L^\nu(\square_k)}
		\|\E v\|_{L^\nu(\square_k)}
		\leq C_Q^+\sum_{k\in\Z^n}
		\|\E u\|_{H^1(\square_k)}
		\|\E v\|_{H^1(\square_k)}
		\\
		&\quad \leq 
		C^+_Q
		\left(\sum_{k\in\Z^n}\|\E u\|^2_{H^1(\square_k)}\right)^{1/2}
		\left(\sum_{k\in\Z^n}\|\E v\|^2_{H^1(\square_k)}\right)^{1/2}=
		C^+_Q\|u\|_{H^1(\Omega)}\|v\|_{H^1(\Omega)}
		\label{Quv:sup}
	\end{align}
	(here we used \eqref{assump:Q2} and the simple bound
	$
		\|Q\|_{L^\mu( \square_k)}\leq
		\|Q\|_{L^{2\mu}( \square_k)}$). Note that the  estimate \eqref{Quv:sup} remains valid if  $Q\in L_{\rm loc}^{\mu}(\R^n)$, 
	$\|{Q}\|_{L^{\mu}(\square_k )}\le C^+_Q$;
	the stronger assumptions \eqref{Q:basic}--\eqref{assump:Q2} are required to
	guarantee $Qf\in L^2(\Omega)$ for any $f\in\HS$, see \eqref{QL2}.
	It follows from \eqref{Quv:sup}  that 
	the   expression for $\|\cdot\|_{\HS}$ is well-defined, moreover, due to \eqref{rho>0}, \eqref{Quv:sup}, we have
	\begin{gather}\label{norm:equivs}
	\exists C_1,C_2>0\ \forall u\in\HS:\quad	C_1\|u\|_{H^1(\Omega)}\leq \|u\|_{\HS}\leq C_2\|u\|_{H^1(\Omega)}.
	\end{gather}

	Let $f\in\HS$. We define the conjugate-linear bounded  functional  
	$F:\HS\to\C$  by
	$$F[v]=\int_{\Omega}{Q}(x)f(x)\overline{v(x)}\d x.$$
	Again, by virtue of the Riesz representation theorem, there is a unique 
	$u\in \HS$ such that
	\begin{gather}\label{variat}
		(u,v)_{\HS}=F[v],\quad \forall v\in \HS.	
	\end{gather}
	We define the operator $\Res:\HS\to\HS$ by $\Res f\ceq u$; this operator is linear, bounded, non-negative and self-adjoint. 
	The function $u=\Res f$ is a weak solution to the problem
	\begin{gather*}
		-\Delta u +{Q} u={Q} f\text{ on }\Omega,\qquad
		u=0\text{ on }\partial\Omega.
	\end{gather*}
	The number $\mu\not=0 $ is an eigenvalue of $\Res$ with the corresponding eigenfunction $u$ iff
	$\lambda=\mu^{-1}-1$ and $u$ are an eigenvalue and an eigenfunction of the 
	  weighted spectral problem
	$$-\Delta u=\lambda Q u\text{ in }\Omega,\quad u=0\text{ on }\partial\Omega.$$

	Since the estimates we will obtain have different forms for $n\ge 3$ and $n=2$, it is convenient to introduce also the   notation
	$$\zeta\ie\ceq\begin{cases}1,&n\ge3,\\|\ln d\ie|,&n=2.\end{cases}$$

	We denote by $\J\e:\HS\to \HS\e$  the operator of restriction to
	$\Omega\e$, i.e. $\J\e f = f\restr_{\Omega\e}$.

	Finally, by $C,C_1,C_2,\dots$ we denote generic positive constants
	being independent of $\eps>0$ and $i\in\I\e$ (but they may depend on $\Omega$ and the constants $\wt C$,   $C_{\rm inn}$, $C_{\rm tr}$,   $C_{\rm P}$,
	$C_{\rm sec}$, $C_Q^\pm$).
	\smallskip

	We are now in position to formulate the first result. Recall that   $\kappa\e$ is given in \eqref{assump:main}.
	
	\begin{theorem}\label{th1}
		Under the  assumptions imposed above the following estimates hold:
		\begin{gather}\label{th1:norms}
			C_1\|f\| _{H^1(\Omega\e)}\le \|f\| _{\HS\e}\le
			C_2\|f\| _{H^1(\Omega\e)},\ \forall f\in\HS\e,
		\end{gather}
		and
		\begin{gather}\label{th1:est}
			\|\Res\e \J\e f  - \J\e\Res f \|_{\H\e}\leq C
			{\,\max\left\{\kappa\e,\,  (\sup_{i\in\I\e}d\ie\zeta\ie)^{1/2}\right\}}\|f\|_{\HS},\quad \forall f\in \HS.
		\end{gather} 
	\end{theorem}
	
	\begin{remark}
	At first glance, it appears that $\rho_\varepsilon$ from \eqref{assump:Te}
	does not influence the estimate \eqref{th1}, which seems suspicious.
	However, this is not the case, since the speed of convergence of $\rho_\varepsilon$
	to zero affects the speed of convergence of $\kappa_\varepsilon$ in \eqref{assump:main}. Indeed, we have
	$$\|Q\e-Q\|^\mu_{L^{\mu}(\square_k)}=
	\|Q\e-Q\|^\mu_{L^{\mu}(\Omega\cap\square_k)}=\|Q\|^\mu_{L^\mu(T\e\cap \square_k)}+\sum_{i\in\I\e}\|Q\e-Q\|^\mu_{L^\mu(\Gamma\ie\cap \square_k)},$$
	and the rate of convergence of the first term $\|Q\|^\mu_{L^\mu(T\e\cap \square_k)}$ is influenced by the ``thickness'' of the set $T_\varepsilon$, which is controlled by  $\rho_\varepsilon$.
\end{remark}

\begin{remark}
	The rate of convergence $\kappa\e$ in \eqref{assump:main} can be arbitrary.
	To demonstrate this, we assume that   $\Omega$ is bounded; for  \emph{constant}  function $Q(x)\equiv 1$, we construct the   holes 
	$D\ie$ as in Remark~\ref{rem:Q} with the following modification:  the expression \eqref{die} for $d\ie$
	we replace by 
	$$
	d\ie = \left(z\e\frac{\wt Q\ie \vol(\Gamma\ie)}{\area(\partial\B(1,0))} \right)^{1/(n-1)}\overset{\eqref{wtQ:def}}=\left(z\e\frac{  \vol(\Gamma\ie)}{\area(\partial\B(1,0))} \right)^{1/(n-1)},
	$$
	with some  $z\e\to 1$ as $\eps\to 0$.
The presence of  $z_\varepsilon$ does not affect the validity of the conditions required for Theorem~\ref{th1}.
	In particular, we have $Q\e(x)=z\e \sum_{i\in\I\e}\mathbb{1}_{\Gamma\ie}(x)$, whence
	$$\|Q\e-Q\|_{L^{\mu}(\Omega)}=
	 \Big(\vol(T\e) +
	|z\e-1|^\mu\sum_{i\in\I\e}\vol(\Gamma\ie)\Big)^{1/\mu} .$$
Furthermore, if $\Omega$ is a convex polytope, it is easy to construct a family of pairwise disjoint convex polytopes $\Gamma\ie$ satisfying \eqref{Gamma:cond:2}, \eqref{assump:re}, and such that $\cup_{i\in\I\e }\overline{\Gamma\ie}=\overline{\Omega}$; in this case we   have 
$$\|Q\e-Q\|_{L^{\mu}(\Omega)}=|z\e-1|(\vol(\Omega))^{1/\mu}.$$
Since the choice of $z_\varepsilon$ is entirely at our disposal,
the factor $|z_\varepsilon-1|$ can converge to zero at an arbitrary rate.  
\end{remark}
	
	Our next goal is to estimate the distance between the spectra of  $\Res\e$ and $\Res$ in the Hausdorff metric.
	Recall that
	for closed sets $X,Y\subset\R$  the \emph{Hausdorff distance} between  them is given by
	\begin{gather*}
		\dist_H (X,Y)\ceq\max\bigg\{\sup_{x\in X} \inf_{y\in Y}|x-y|;\,\sup_{y\in Y} \inf_{x\in X}|y-x|\bigg\}.
	\end{gather*}
	One can easily show  (see, e.g., \cite[Proposition~A.1.6]{P12}) that $\dist_H(X\e,X)\to 0$ if and only if   the following two conditions hold simultaneously:
	\begin{itemize}
		\item[(i)] $\forall x\in X$ there exists a family $(x\e)_{\eps>0}$ with $x\e\in X\e$ such that $x\e\to x$ as $\eps\to 0 $,\smallskip
		
		\item[(ii)] $\forall x\in\R\setminus X$ there exist $\delta>0$ such that $X\e\cap (x-\delta,x+\delta)=\varnothing$ for  
		small enough  $\eps$.
		
	\end{itemize}

	\begin{theorem}\label{th2}
		One has the estimate
		\begin{gather} \label{th2:est}
			\dist_H\left(\sigma(\Res\e),\sigma(\Res)\right)\leq 
			C		{\,\max\left\{\kappa\e,\,  (\sup_{i\in\I\e}d\ie\zeta\ie)^{1/2}\right\}}.
		\end{gather}
		
	\end{theorem} 
	
	The above theorems will be proven in Sections~\ref{sec3}-\ref{sec4}.
	
	\subsection{Remarks on holes shape assumptions} 
	We conclude this section with a more detailed discussion of assumptions
	imposed on the shapes of 
	the holes.
	Condition \eqref{assump:shape1} is straightforward -- it prevents the use of very thin domains $\D\ie$, whereas conditions \eqref{assump:shape2}--\eqref{assump:shape4} are more subtle (apart from the trivial case, when there is a \emph{finite} set of domains such that every 
	$\D\ie$ is congruent to one of them). 
	\smallskip
	
	$(i)$	The standard proof
	of the trace inequality (see, e.g., \cite[\S 5.5, Theorem 1]{E10}) is based on a local straightening of the boundary of a domain. 
	Tracing this proof, one can easily show that the assumption 
	\eqref{assump:shape2} holds if the principle curvatures of each $\partial D\ie$ are bounded uniformly in $\eps>0$ and $i\in\I\e$, and moreover, 
	there is $\delta>0$ such that the $\delta$-neighborhood 
	of each $\partial \D\ie$ possesses global tubular coordinates.
	\smallskip
	
	$(ii)$ We next proceed to assumption \eqref{assump:shape3}. 
	According to \cite[Corollary 2]{BC07}, a sufficient condition for \eqref{assump:shape3} to hold is that all domains $\mathbf{G}\ie$ are uniformly Lipschitz regular (see \cite[Definition~I.1]{C77}). This condition is satisfied, for instance, if all domains $\mathbf{G}_i$ obey the cone property with a fixed cone; that is, for every boundary point of $\mathbf{G}\ie$ one can associate a homothetic copy of a cone with fixed angle and height that is contained in the domain. For details, see Definitions 2 and 3 in \cite{C75} and the remark following Definition~I.1 in \cite{C77}.
	
	The preceding discussion implies that assumption \eqref{assump:shape3} holds whenever all sets $\D\ie$ are \emph{convex}. Indeed, in this case, the sets $\mathbf{G}\ie$ satisfy the cone property with a fixed cone of arbitrary angle $0<\al<\pi/2$ and sufficiently small height $h(\alpha)$ (here, in addition to convexity, we also use the inclusion $\D\ie \subset \B(1,0)$).
	
	If one wishes to obtain explicit lower bounds for the Neumann eigenvalues $\Lambda_N(\mathbf{G}\ie)$ in terms of the geometry of $\mathbf{G}\ie$, one can use, for example, the estimate \cite[Theorem~1.1]{BCT15}
	\begin{equation*}
		C_n \lambda\e^\sharp K_{n,i,\varepsilon}^2\le\Lambda_{\rm N}(\mathbf{G}\ie).
	\end{equation*}
	Here $C_n>0$ depends only on the dimension $n$, $\lambda\e^\sharp$ is the first Dirichlet eigenvalue of a ball with the same measure as $\mathbf{G}\ie$, and $K_{i,\varepsilon}$ is the best isoperimetric constant relative to $\mathbf{G}\ie$, i.e.
	\begin{equation*}
		K_{i,\varepsilon} = \inf_{E \subset \mathbf{G}\ie} 
		\frac{P{\ie} (E)}{( \min \{ \vol(E), \vol(\mathbf{G}\ie \setminus E)
			\} )^{(n-1)/n} }
	\end{equation*}
	where $P{\ie}(E)$ denotes the perimeter of $E$ relative to $\mathbf{G}\ie$. The inclusion $\mathbf{G}\ie\subset\B(2,0)$
	and the domain monotonicity of Dirichlet eigenvalues imply
	that  $\lambda\e^\sharp $ are bounded from below by the first  eigenvalue $\Lambda_{\rm D}(\B(2,0))$ of the Dirichlet Laplacian on  $\B(2,0)$.
	Thus, if  the isoperimetric constants $K\ie$ are bounded from below by $K>0$, then
	\eqref{assump:shape3} holds with $\Lambda_N=C_n \Lambda_{\rm D}(\B(2,0)) K^2$.
	
	Finally, for $n=2$ we give an example of  $\mathbf{D}\ie$ whose smallest enclosing ball is $\B(1,0)$ and 
	\begin{gather}\label{inf0}
		\inf_{\eps>0}\inf_{i\in\I\e}\Lambda_N(\mathbf{G}\ie)=0,\quad
		\mathbf{G}\ie=
		\B(2,0)\setminus \overline{\D\ie}.
	\end{gather}
	We define the sets $\D\ie$ by
	$$
	{\D\ie} = (\B(1,0)\setminus \overline{\B(1/2,0)})\setminus \mathbf{F}\e,\text{ where } \mathbf{F}\e\ceq \left\{ 
	(r \cos \theta, r \sin \theta)\in\R^2  : \  r\in (1/2,1),\, 
	|\theta| < \beta\e \right\}.
	$$
	Here   $\beta\e \in (0,\pi/2]$ such that
	$\lim_{\eps\to 0} \beta\e = 0.$  
	Evidently, the smallest ball containing $\mathbf{D}\ie$ is indeed $\B(1,0)$.
	To prove \eqref{inf0}, we introduce the test-function $U\e\in H^1(\mathbf{G}\ie)$ by 
	$$
	U\e(x)\ceq 
	\begin{cases}
		1,& x\in \B(1/2,0),\\
		C\e,& x\in \B(2,0)\setminus\overline{\B(1,0)},\\
		\frac{(C\e-1)\ln|x|}{\ln 2}+C\e,& x\in \mathbf{F}\e,
	\end{cases}
	$$
	where the constant $C\e$ is chosen to satisfy 
	\begin{gather}\label{mean0}
		\int_{\mathbf{G}\ie}U\e(x)\d x=0. 
	\end{gather}
	Simple computations
	yield
	$\lim_{\eps\to 0}C\e=
	- {\vol(\B(1/2,0))}/ {\vol(\B(2,0)\setminus {\B(1,0)})},$
	whence we get
	\begin{gather}\label{Ue}
		\lim_{\eps\to 0}\|U\e\|^2_{L^2(\mathbf{G}\ie)}>0,
		\quad
		\lim_{\eps\to 0}\|\nabla U\e\|^2_{L^2(\mathbf{G}\ie)}=0.
	\end{gather}
	Then \eqref{inf0} follows from \eqref{mean0}--\eqref{Ue}
	and the min-max principle.
	Similar examples can be constructed for $n\ge 3$.
	\smallskip

	$(iii)$ To guarantee the fulfillment of assumption \eqref{assump:shape4},
	we assume that there exist constants $\rho,\delta > 0$ (independent on $\eps$ and $i$) such that 
	each set  $\mathbf{G}\ie$ is a so-called  $(\rho,\delta)$-domain in the 
	sense of Jones \cite{Jones}. Without going into the details of this definition, let us note that this property holds when all domains $\mathbf{G}\ie$ are uniformly Lipschitz-regular as in  $(ii)$. 
	Under the above assumption on the sets $\mathbf{G}\ie$,
	Theorem~1.1 from \cite{Jones} ensures the existence of  extension operators $\mathbf{E}\ie: H^1(\mathbf{G}\ie) \to H^1(\mathbb{R}^n)$, whose norms are bounded by a constant $C_E$ independent of $i$ and $\varepsilon$ (see \cite{Jones,Rogers} and the references therein for these and more general results). However, the extensions $\mathbf{E}\ie $ are not, in general, harmonic on $\D\ie$; therefore, they must be modified in order to satisfy the requirements of assumption \eqref{assump:shape4}. 
	
	Given $i, \varepsilon$ and $u \in  H^1(\mathbf{G}\ie)$, we define 
	$f\ie = \Delta (\mathbf{E}_{i,\varepsilon}u)$. Then, $f\ie \in 
	H^{-1}(\mathbb{R}^n)$, and 
	$$
	\Vert f\ie \Vert_{H^{-1}(\mathbb{R}^n)} \leq \Vert \mathbf{E}_{i,\varepsilon}u \Vert_{H^{1}(\mathbb{R}^n)} \leq  C_E \Vert u \Vert_{H^1(\mathbf{G}\ie)}.
	$$ Next, let $w\ie\in H_0^1({\D\ie})$
	be the weak solution to the problem
	$$
	-\Delta w\ie=f\ie\text{ in }\D\ie,\quad w\ie=0\text{ on }\partial\D\ie
	$$
	(here we regard $f\ie$ as an element of $H^{-1}(\D\ie)$ via the canonical embedding $H^{-1}(\R^n)\hookrightarrow H^{-1}(\D\ie)$).
	One has the estimates
	\begin{gather*}
		\Lambda_{\rm D}(\D\ie)\|  w\ie\|^2_{L^2(\D\ie)}\leq \|\nabla w\ie\|^2_{L^2(\D\ie)}\leq  \|w\ie\|_{H^{1}(\D\ie)}\|f\ie\|_{H^{-1}(\D\ie)},
	\end{gather*}
	where $\Lambda_{\rm D}(\D\ie)$ is the lowest  eigenvalue of the Dirichlet Laplacian on $\D\ie$. The domain monotonicity of  Dirichlet eigenvalues 
	implies that $\Lambda_{\rm D}(\D\ie)$ are bounded from below by the 
	lowest  eigenvalue $\Lambda_{\rm D}(\B(1,0))$ of the Dirichlet Laplacian on $\B(1,0)$. Thus, we obtain
	\begin{gather}\label{wie:est}
		\|w\ie\|_{H^1(\D\ie)}\leq C_E\big(1+(\Lambda_{\rm D}(\B(1,0)))^{-1}\big)\|u\|_{H^1(\mathbf{G}\ie)}.
	\end{gather}
	Finally, we define the function   $U\ie\in H^1(\B(2,0))$
	by
	\begin{equation*}
		U\ie(x)= \left\{
		\begin{array}{ll}
			u(x), & x \in \mathbf{G}\ie, \\
			\mathbf{E}_{i,\varepsilon} u (x) + w\ie(x) , & x \in {\D\ie}.
		\end{array}
		\right.
	\end{equation*}
	By construction, $\Delta U\ie=0$ in $ {\D\ie}$, whence $U\ie =\mathbf{P}\ie u$.
	The fulfilment of \eqref{assump:shape4} then follows from 
	\eqref{wie:est} together with the bound $\|\mathbf{E}\ie\|\leq C_E$.

	\section{Auxiliary estimates\label{sec3}}
	
	In this section we collect several estimates which will be used in the proofs of the main theorems. 
	At first, we introduce several new notations:
	\begin{itemize} \setlength\itemsep{0.5em}
		
		\item $B\ie\ceq \B(d\ie,x\ie),$
		
		\item $F\ie \ceq \B(2d\ie,x\ie)$,
		
		\item $R\ie\ceq \B(C_{\rm sec}r\ie,x\ie),$ cf.~\eqref{assump:D1}.

	\end{itemize}
	Note that in view of \eqref{assump:D1+}, \eqref{assump:D1}, we get
	$$D\ie\subset B\ie\subset F\ie\subset R\ie\subset\Gamma\ie.$$

	We denote by  $\la u\ra_{D}$   the mean value of a function $u$ in a domain $D$, i.e.  
	$$
	\la u\ra_{D}\ceq  \frac1{\vol(D)}\int_{D}u\d x.
	$$ 
	The same notation is used for the mean value of $u$ in an 
	$(n-1)$-dimensional manifold $S$:
	$$
	\la u\ra_{S}\ceq  \frac1{\area(S)}\int_{S}u\d s.
	$$ 
	
	Recall that the {symbol} $\Lambda_{\rm N}(D)$ stands for the smallest non-zero eigenvalue of the Neumann Laplacian on $D\subset\R^n$.
	Also, by $\Lambda_{R}^\al(D)$ we denote the smallest eigenvalue of the Laplacian on  $D$ subject to the Robin boundary conditions $\frac{\partial u}{\partial\nu}+\al u=0,$
	where ${\partial\over\partial \nu}$ stands for the derivative along the
	exterior unit normal. The min-max principle yields the estimate
	\begin{gather} 
		\label{Neumann:est:general}
		\|u-\la u \ra_{D}\|^2_{L^2(D)}\leq 
		(\Lambda_{\rm N}(D))^{-1}
		\|\nabla u\|^2_{L^2(D)},\ \forall u\in H^1(D),	 
		\\
		\label{Robin:est:general}
		\|u\|^2_{L^2(D)}\leq 
		(\Lambda_{\rm R}^1(D))^{-1}\left(\|u\|^2_{L^2(\partial D)}+
		\|\nabla u\|^2_{L^2(D)}\right),\ \forall u\in H^1(D).	 
	\end{gather}

	In the following, we assume that $\eps$ is sufficiently small\footnote{Recall that the small parameter $\eps$ takes discrete values $\eps=m^{-1}$, $m\in\N$, thus it is enough to establish the estimates 
		\eqref{th1:norms}, \eqref{th1:est}, \eqref{th2:est} only for small $\eps$.}	
	in order to have 
	\begin{align}
		\label{eps0:1b}
		  C_{\rm sec}r\ie\leq 1/2,\ \forall i\in\I\e
	\end{align}
	(the inequalities \eqref{eps0:1b} hold for sufficiently small $\eps$ due to  \eqref{assump:re}).

	The first estimate was established in \cite{MK06}.

	\begin{lemma}[\!{\cite[Lemma~4.9 \& Remark~4.2]{MK06}}]
		\label{lemma:MK06}
		Let $D\subset\R^n$ be a convex domain, and let
		$D_1,D_2$ be other two domains contained in $D$, moreover, $\vol(D_2)\not= 0$.  Then
		\begin{equation*}\forall u\in  H^1(D):\quad
			\|u\|^2_{L^2(D_1)}
			\leq \frac {2\vol(D_1)}{\vol(D_2)}\|u\|^2_{L^2(D_2)} +
			C \frac {(\mathrm{diam}(D))^{n+1}(\vol(D_1))^{1/n}}{\vol(D_2)}\|\nabla u\|^2_{L^2(D)},
		\end{equation*}
		where the constant $C>0$ depends only on the dimension $n$.
	\end{lemma}
	
	In all lemmas below, we do not assume that the assumptions \eqref{assump:Q1}, \eqref{assump:main} 
	are satisfied; we only claim the fulfillment of 
	\eqref{Gamma:cond:2}, \eqref{assump:D1}, \eqref{assump:D1+}, and  the 
	assumptions \eqref{assump:shape1}--\eqref{assump:shape4} on the shapes of $\D\ie$.
	We believe that these lemmas are of independent interest and may be useful in future research.

	\begin{lemma}\label{lemma1}
		One has the estimate
		\begin{gather}  \label{lemma1:est}
			\|u\|^2_{L^2(\partial D\ie)}\leq 
			C\left(d\ie^{n-1}r\ie^{-n}\|u\|^2_{L^2(R\ie\setminus\overline{B\ie})}+d\ie\zeta\ie\|\nabla u\|^2_{L^2(R\ie)}\right),\ \forall u\in H^1(R\ie).
		\end{gather}  
	\end{lemma}
	
	\begin{remark}A similar estimate to \eqref{lemma1:est} was established in \cite[Lemma~3.3]{GL21}. Compared to our result, 
		the coefficient in front of $\|\nabla u\|^2_{L^2(R_{i,\varepsilon})}$   is larger there, namely 
	$r\ie$ instead of $d\ie\zeta\ie$ (in our notation).
	\end{remark}

	\begin{proof}[Proof of Lemma~\ref{lemma1}]
		Using the coordinate transformation 
		\begin{gather}\label{coord}
			y\mapsto x=d\ie y +x\ie, 
		\end{gather}		
		which maps 
		$\D\ie$ to $D\ie$, we derive from \eqref{assump:shape2}:
		\begin{gather}\label{lemma1:1}
			\|u\|^2_{L^2(\partial D\ie)}\leq 
			C_{\rm tr}^2\left(d\ie^{-1}\|u\|_{L^2(D\ie)}^2+d\ie\|\nabla u\|_{L^2(D\ie)}^2\right),\
			\forall u\in H^1(D\ie).
		\end{gather}	
		
		Using \eqref{Robin:est:general} for $D=\B(1,0)$
		and then again applying the coordinate transformation \eqref{coord} (it maps $\B(1,0)$ to $B\ie$), we get the estimate
		\begin{align}
			\notag
			\|u\|^2_{L^2(D\ie)}&\leq
			\|u\|^2_{L^2(B\ie)}\\ \label{lemma1:2} &\leq (\Lambda_{\rm R}^1(\B(1,0)))^{-1}\left(d\ie\|u\|^2_{L^2(\partial B\ie)}+
			d\ie^2\|\nabla u\|^2_{L^2(B\ie)}\right),\ \forall u\in H^1(B\ie).
		\end{align}	
		
		To proceed further, we introduce the spherical coordinates $(r,\phi)$ in 
		$\overline{R\ie\setminus B\ie}$ with a pole at $x\ie$.
		Here $\phi=(\phi_1,{\dots},\phi_{n-1})$ are the angular coordinates ($\phi_j\in [0,\pi]$ for $j=1,\dots,n-2$, $\phi_{n-1}\in [0,2\pi)$) and
		$r\in [d\ie,C_{\rm sec}r\ie]$ is the radial coordinate. 
		One has
		\begin{gather}\label{FTC}
			u(d\ie,\phi)=u(r,\phi)-\int_{d\ie}^{r}{\partial
				u(\tau,\phi)\over\partial\tau}\d\tau,\ r\in (d\e,C_{\rm sec}r\ie),
		\end{gather}
		whence, using the Cauchy-Schwarz,	
		we deduce
		\begin{align} \notag
			|u(d\ie,\phi)|^2&\le 2|u(r,\phi)|^2+2\left|\int_{d\ie}^{r}{\partial
				u(\tau,\phi)\over\partial\tau}\d\tau\right|^2\\
			&\leq
			2|u(r,\phi)|^2+
			2\int_{d\ie}^{C_{\rm sec}r\ie}\tau^{1-n}\d\tau \int_{d\ie}^{C_{\rm sec}r\ie}\left|{\partial
				u(\tau,\phi)\over\partial\tau}\right|^2\tau^{n-1}\d\tau\notag\\
			&\leq 2|u(r,\phi)|^2+
			Cd\ie^{2-n}\zeta\ie \int_{d\ie}^{C_{\rm sec}r\ie}\left|{\partial
				u(\tau,\phi)\over\partial\tau}\right|^2\tau^{n-1}\d\tau\label{lemma1:3} 
		\end{align} 
		(on the last step in the case $n=2$ we used $\ln(C_{\rm sec}r\ie)<0$, which follows from \eqref{eps0:1b}).
		Multiplying \eqref{lemma1:3} by 
		$$\left(\ds\int_{d\ie}^{C_{\rm sec}r\ie} r^{n-1}\d r\right)^{-1}d\ie^{n-1}r^{n-1}\prod_{j=1}^{n-2}\left(\sin\phi_j\right)^{n-1-j},$$ 
		and   integrating
		over $r\in (d\ie,C_{\rm sec}r\ie)$, $\phi_j\in (0,\pi)$, $j=1,\dots,n-2$, 
		$\phi_{n-1}\in (0,2\pi)$, 
		we get
		\begin{equation*}
			\|u\|^2_{L^2(\partial B\ie))}
			\leq C\left(d\ie^{n-1}\left(\ds \int_{d\ie}^{C_{\rm sec}r\ie} r^{n-1}\d r\right)^{-1}\|u\|^2_{L^2(R\ie\setminus \overline{B\ie})}
			+
			d\ie\zeta\ie\|\nabla u\|^2_{L^2(R\ie\setminus \overline{B\ie})}\right).
		\end{equation*}
		It follows from \eqref{assump:D1+} that  
		$
		\ds	C r\ie^n\leq \int_{d\ie}^{C_{\rm sec}r\ie} r^{n-1}\d r
		$. Hence, we get
		\begin{equation}\label{lemma1:4}
			\|u\|^2_{L^2(\partial B\ie))}
			\leq C \left(d\ie^{n-1} r\ie^{-n}\|u\|^2_{L^2(R\ie\setminus \overline{B\ie})}
			+
			d\ie\zeta\ie\|\nabla u\|^2_{L^2(R\ie\setminus \overline{B\ie})}\right).
		\end{equation}
		
		Combining \eqref{lemma1:1}, \eqref{lemma1:2}, \eqref{lemma1:4}, and taking into account that $1<\zeta\ie$
		(the latter bound follows easily from \eqref{assump:D1+} and \eqref{eps0:1b}, and is required only in the case $n=2$), 	we arrive at the  required estimate \eqref{lemma1:est}.
		The lemma is proven. 	
	\end{proof}

	\begin{lemma}\label{lemma1+}
		One has the estimate
		\begin{gather}  \label{lemma1+:est}
			\|u\|^2_{L^2(D\ie)}\leq 
			C\left(d\ie^{n}r\ie^{-n}\|u\|^2_{L^2(R\ie\setminus\overline{B\ie})}+d\ie^2\zeta\ie\|\nabla u\|^2_{L^2(R\ie)}\right),\ \forall u\in H^1(R\ie).
		\end{gather} 
	\end{lemma}
	
	\begin{proof}
		The estimate \eqref{lemma1+:est} immediately follows from  \eqref{lemma1:2}, \eqref{lemma1:4} and $1<\zeta\ie$ (as in the previous lemma, the latter is needed only in the case $n=2$).
	\end{proof}

	\begin{lemma}\label{lemma2}
		One has the estimate
		\begin{gather}\label{lemma2:est}
			\|u\|^2_{L^2(\Gamma\ie)}\leq 
			C\left( 
			r\ie^{n}d\ie^{1-n}\|u\|^2_{L^2(\partial D\ie)}+r\ie^{n}d\ie^{2-n}  \zeta\ie\|\nabla u\|^2_{L^2(\Gamma\ie)}\right),\ \forall u\in H^1(\Gamma\ie).
		\end{gather}
		
	\end{lemma}
	
	\begin{proof}
		Applying Lemma~\ref{lemma:MK06} with $D=\Gamma\ie$ (recall that the sets $\Gamma\ie$ are  convex polytops), $D_1=\Gamma\ie$, $D_2=R\ie\setminus\overline{B\ie}$ and taking into account \eqref{Gamma:cond:2} and \eqref{assump:D1+}, we easily get
		\begin{gather}\label{lemma2:1}
			\|u\|^2_{L^2(\Gamma\ie)}\leq C\left(\|u\|^2_{L^2(R\ie\setminus\overline{B\ie})}+r\ie^2\|\nabla u \|^2_{L^2(\Gamma\ie)}\right),\ \forall u\in H^1(\Gamma\ie).
		\end{gather}	 
		
		We also have the estimate
		\begin{equation}\label{lemma2:2}
			\|u\|^2_{L^2(R\ie\setminus \overline{B\ie})}
			\leq C \left(r\ie^{n}d\ie^{1-n}\|u\|^2_{L^2(\partial B\ie)}
			+
			r\ie^nd\ie^{2-n}\zeta\ie\|\nabla u\|^2_{L^2(R\ie\setminus \overline{B\ie})}\right),\ \forall u\in H^1(R\ie\setminus \overline{B\ie}).
		\end{equation}
		Its proof is similar to the proof of  \eqref{lemma1:4} and we left it to the reader
		(one has to introduce the polar coordinates $(r,\phi)$ in $\overline{R_\varepsilon \setminus B_\varepsilon}$
		and apply the fundamental theorem of calculus (cf.\eqref{FTC}), but now with $u(r,\phi)$ on the left-hand side and $u(d\e,\phi)$ on the right-hand side).
		
		Via the coordinate transformation \eqref{coord} mapping
		$\B(1,0)$ to $B\ie$, we infer
		from the trace estimate 
		$\|u\|_{L^2(\partial \B(1,0))}\leq C\|u\|_{H^1(\B(1,0))}$:
		\begin{gather} \label{lemma2:3}
			\|u\|^2_{L^2(\partial B\ie)}\leq 
			C\left(d\ie^{-1}\|u\|_{L^2(B\ie)}^2+d\ie\|\nabla u\|_{L^2(B\ie)}^2\right),\
			\forall u\in H^1(B\ie).
		\end{gather}	  
		
		Applying Lemma~\ref{lemma:MK06} with $D=D_1=B\ie$, $D_2=D\ie$ and taking into account \eqref{assump:shape1}, we get
		\begin{gather}\label{lemma2:4}
			\|u\|^2_{L^2(B\ie)}\leq 
			C\left(\|u\|_{L^2(D\ie)}^2+d\ie^2\|\nabla u\|_{L^2(B\ie)}^2\right),\
			\forall u\in H^1(B\ie).
		\end{gather}

		Finally, via the coordinate transformation \eqref{coord} mapping
		$\D\ie$ to $D\ie$ and the bound \eqref{Robin:est:general} being applied for 
		$D=\D\ie$, we  deduce 
		\begin{gather}\label{lemma2:5}
			\|u\|^2_{L^2(D\ie)}\leq 
			(\Lambda_{\rm R}^1(\D\ie))^{-1}\left(d\ie\|u\|^2_{L^2(\partial D\ie)}+d\ie^2\|\nabla u\|^2_{L^2(D\ie)}\right),\ \forall u\in H^1(D\ie).
		\end{gather}	
		The Faber-Krahn inequality for Robin problems  \cite[Theorem~1.1]{Da06} asserts that amongst all Lipschitz domains of fixed volume the
		ball has the smallest first Robin  eigenvalue. Hence
		\begin{gather*}
			\Lambda_{\rm R}^1(\B(\ell\ie,0)) \leq \Lambda_{\rm R}^1(\D\ie),\text{\quad where }
			\ell\ie\ceq \big(\vol(\D\ie)/\vol(\B(1,0))\big)^{1/n}. 
		\end{gather*}
		Due to \eqref{assump:shape1}, we get
		$C_{\rm inn} \leq \ell\ie \leq 1$,
		whence, using simple rescaling arguments and 
		the fact that the first Robin eigenvalue $\Lambda_R^\al(D)$
		is monotonically increasing with respect to $\al$,
		we conclude
		\begin{gather}\label{lemma2:6}
			\Lambda^1_{\rm R}(\B(\ell\ie,0))=
			\ell\ie^{-2}\Lambda^{ \ell\ie}_{\rm R}(\B(1,0))\geq 
			\ell\ie^{-2}\Lambda^{C_{\rm inn}}_{\rm R}(\B(1,0))\geq \Lambda^{C_{\rm inn}}_{\rm R}(\B(1,0)).
		\end{gather}
		From \eqref{lemma2:5} and \eqref{lemma2:6}, we conclude	
		\begin{gather}\label{lemma2:7}
			\|u\|^2_{L^2(D\ie)}\leq C\left(d\e\|u\|^2_{L^2(\partial D\ie)}+d\e^2\|\nabla u\|^2_{L^2(D\ie)}\right),\ 
			\forall u\in H^1(D\ie).
		\end{gather} 
		
		Combining \eqref{lemma2:1}--\eqref{lemma2:4}, \eqref{lemma2:7}, the bound $1<\zeta\ie$ (as in the previous lemmas, we need it only in the case $n=2$) and the bound
		$r\ie^2 \le Cr\ie^n d\ie^{2-n}$ (it follows from \eqref{assump:D1+})
		we arrive at the desired estimate \eqref{lemma2:est}. The lemma is proven.
	\end{proof}

	\begin{lemma}\label{lemma3}
		One has the estimate
		\begin{gather}  \label{lemma3:est}
			|\la u\ra_{\partial  D\ie}-\la u\ra_{\Gamma\ie}|^2\leq 
			Cd\ie^{2-n}\zeta\ie\|\nabla u\|^2_{L^2(\Gamma\ie)},\ \forall u\in H^1(\Gamma\ie).
		\end{gather} 
	\end{lemma}
	
	\begin{proof}
		One has:
		\begin{align}  \notag
			|\la u\ra_{\partial  D\ie}-\la u\ra_{B\ie}|^2&\leq 
			(\area(\partial D\ie))^{-1}\|u-\la u\ra_{B\ie}\|^2_{L^2(\partial D\ie)}
			\\\notag
			&\leq
			C  (\area(\partial D\ie))^{-1}
			\left(d\ie^{-1}\|u-\la u\ra_{B\ie}\|^2_{L^2( D\ie)}+d\ie\|\nabla u\|^2_{L^2(D\ie)}\right)\\ \label{lm3:1}
			&\leq
			C(\area(\partial D\ie))^{-1}
			d\ie\|\nabla u\|^2_{L^2(B\ie)},\quad \forall u\in H^1(B\ie),
		\end{align}
		where on the first step we applied the Cauchy-Schwarz inequality,
		on the second step we used \eqref{lemma1:1}, and on	
		the last step we utilized the inclusion 
		$D\ie\subset B\ie$ and the Poincar\'e inequality
		$$ 
		\|u-\la u\ra_{B\ie}\|^2_{L^2( B\ie)}\leq
		d\ie^2 (\Lambda_{\rm N}(\B(0,1)))^{-1}\|\nabla u\|_{L^2(B\ie)}^2,\ \forall u\in H^1(B\ie).
		$$
		The latter estimate follows from \eqref{Neumann:est:general} (applied for $D=\B(1,0)$)
		by virtue of the coordinate transformation \eqref{coord}.
		Also, using the isoperimetric inequality
		\begin{gather}\label{isoper}
			\area(\partial \D\ie) \geq n (\vol(\D\ie))^{(n-1)/n}(\vol(\B(1,0)))^{1/n},
		\end{gather}
		and \eqref{assump:shape1}, we obtain
		\begin{gather}\label{area:est+}
			\area(\partial D\ie)=d\ie^{n-1}\area(\partial \D\ie)\geq Cd\ie^{n-1}.
		\end{gather}
		Combining \eqref{lm3:1} and \eqref{area:est+}, we get
		\begin{gather} \label{lemma3:1} 
			|\la u\ra_{\partial  D\ie}-\la u\ra_{B\ie}|^2 
			\leq
			C d\ie^{2-n}\|\nabla u\|^2_{L^2(B\ie)},\ \forall u\in H^1(B\ie).
		\end{gather}
		By a similar argument, we derive the bound
		\begin{gather} \label{lemma3:1+} 
			|\la u\ra_{\partial  B\ie}-\la u\ra_{B\ie}|^2 
			\leq
			C d\ie^{2-n}\|\nabla u\|^2_{L^2(B\ie)},\ \forall u\in H^1(B\ie).
		\end{gather}
		
		Next, we have the estimate
		\begin{gather}  \label{lemma3:2}
			|\la u\ra_{\partial  B\ie}-\la u\ra_{R\ie\setminus \overline{B\ie}}|^2 
			\leq
			C d\ie^{2-n}\zeta\ie\|\nabla u\|^2_{L^2(R\ie\setminus \overline{B\ie})},\
			\forall u\in H^1(R\ie\setminus \overline{B\ie}).
		\end{gather} 
		For the proof of \eqref{lemma3:2}, see \cite[Lemma~2.1.I]{Kh13}; it follows the same line of argument as the proofs of \eqref{lemma1:4} and \eqref{lemma2:2}.
		
		Finally, one has
		\begin{align}\notag
			|\la u\ra_{R\ie\setminus \overline{B\ie}}-\la u\ra_{\Gamma\ie}|^2=
			|\la u-\la u\ra_{\Gamma\ie}\ra_{R\ie\setminus \overline{B\ie}}|^2 
			&\leq	
			(\vol(R\ie\setminus B\ie))^{-1}\|u-\la u\ra_{\Gamma\ie}\|^2_{L^2(R\ie\setminus \overline{B\ie})}\\
			\notag
			&\leq
			(\vol(R\ie\setminus B\ie))^{-1}\|u-\la u\ra_{\Gamma\ie}\|^2_{L^2(\Gamma\ie)}
			\\
			\label{lemma3:3}
			&\leq
			\big(\vol(R\ie\setminus B\ie)\Lambda_{\rm N}(\Gamma\ie)\big)^{-1}\|\nabla u\|^2_{L^2(\Gamma\ie)}
		\end{align}
		Since the sets $\Gamma\ie$ are assumed to be convex, we have the  estimate
		\begin{gather}
			\label{LambdaN:Gamma:est}
			\Lambda_{\rm N}(\Gamma\ie)\geq  {\pi^2}(2\wt  r\ie)^{-2},
		\end{gather}
		which follows from 	 the following Payne-Weinberger bound \cite{PW60}\footnote{Note that the original proof in \cite{PW60} had a mistake for $n\ge 3$; it was later corrected in \cite{B03}.}:
		$$
		\Omega\text{ is convex domain}\ \Longrightarrow\ \Lambda_{\rm N}(\Omega)\geq {\pi^2}(\diam  \Omega)^{-2}$$
		and the bound $\diam(\Gamma\ie)\le 2\wt r\ie$.
		Due to \eqref{assump:D1+}, we have 
		\begin{gather}\label{volRB}
			\vol(R\ie\setminus B\ie)\ge C r\ie^n.
		\end{gather}
		From \eqref{lemma3:3}--\eqref{volRB} and \eqref{Gamma:cond:2}, we conclude the estimate
		\begin{align}\label{lemma3:4}
			|\la u\ra_{R\ie\setminus \overline{B\ie}}-\la u\ra_{\Gamma\ie}|^2 \leq 
			Cr\ie^{2-n}\|\nabla u\|^2_{L^2(\Gamma\ie)}.
		\end{align}
		Combining \eqref{lemma3:1}--\eqref{lemma3:2}, \eqref{lemma3:4}, and taking into account that
		$1<\zeta\ie$ (we need the latter bound only in the case $n=2$) and \eqref{assump:D1+}, we arrive
		at    \eqref{lemma3:est}. The lemma is proven.
	\end{proof}
	
	Recall that   $T\e$ and $\rho\e$ are given in 
	\eqref{re:Te:def} and \eqref{assump:Te}, respectively.
	
	\begin{lemma} \label{lemma4}
		One has the  estimate
		\begin{gather} \label{lemma4:est} 
			\|u\|_{L^2(T\e)}\leq 
			C \rho\e \|\nabla u\|_{L^2(\Omega)},\ \forall u\in H^1_0(\Omega).
		\end{gather} 
	\end{lemma}
	
	\begin{proof}
		Recall that $T\e$ is a subset of the set $$\wh T\e\ceq \{y= x - t \nu(x)\in\R^n:\  x\in\partial\Omega,\ t\in (0,\,\rho\e)\}.$$
		We denote by $\A\e$ the Laplace operator on $\widehat{T}\e$ subject to the Dirichlet conditions on $\partial\Omega$ and the Neumann conditions on $\partial \widehat{T}\e\setminus\partial\Omega$. 
		One has the following asymptotics \cite{Kr14}:
		\begin{equation}\label{lemma4:1}
			\inf \sigma(\A\e)
			=\left(\frac{\pi}{2  \rho\e}\right)^2+\frac{C_{\rm curv}}{\rho\e} +\mathcal{O}(1),\ \rho\e\to 0,\end{equation}
		where  $C_{\rm curv}$ is a constant depending on the principal curvatures of $\partial\Omega$.
		The validity of \eqref{lemma4:1} requires 
		the map~\eqref{tau} to be one-to-one on 
		$\partial\Omega \times (0,\,\rho\e)$, which is indeed the case.
		
		Using the minimax principle, we infer from \eqref{lemma4:1}: 
		\begin{gather}\label{lemma4:2}
			\|u\|_{L^2(\widehat{T}\e)}\leq C \rho\e\|\nabla u\|_{L^2(\widehat{T}\e)},\ \forall
			u\in H^1(\widehat{T}\e)\text{ with }u\!\restr_{\partial\Omega}=0.
		\end{gather}
		Combining \eqref{lemma4:2} and the inclusions $T\e\subset\wh T\e\subset\Omega$, we arrive at
		\eqref{lemma4:est}. The lemma is proven.
	\end{proof}
	
	The last result concerns a suitable  operator   of extension  from $H^1(\Omega\e)$
	to $H^1(\Omega)$. First, we introduce such an operator in the vicinity of each hole 
	$D\ie$. 
	Recall that $F\ie = \B(2d\ie,x\ie)$ and the operator $\mathbf{P}\ie$ is defined in Subsection~\ref{subsec:22}.
	We  define $\P\ie:H^1(F\ie\setminus \overline{D\ie})\to H^1(F\ie)$ as follows,
	$$(\P\ie u)(x)=(\mathbf{P}\ie \mathbf{u})((x-x\ie)d\ie^{-1}),\ x\in F\ie ,$$
	where $\mathbf{u}\in H^1(\mathbf{G}\ie)$ is given by
	$\mathbf{u}(y)=u(d\ie y+x\ie)$;
	recall that $\mathbf{G}\ie=\B(2,0)\setminus\overline{\D\ie}$.  
	Finally, we introduce the operator $\P\e:H^1(\Omega\e)\to H^1(\Omega)$ by
	\begin{gather}\label{Pe}
		(\P\e u)(x)=
		\begin{cases}
			u(x),&x\in\Omega\e,
			\\
			(\P\ie u)(x),&x\in D\ie,
			\ i\in\I\e.
		\end{cases}
	\end{gather}
	Evidently, one has $\Delta (\P\e u)=0\text{ on }\cup_{i\in\I\e}D\ie.$

	\begin{lemma}
		\label{lemma5}
		One has the estimates:
		\begin{align}\label{lemma5:est1}
			\|\nabla (\P\e u)\|_{L^2(\Omega)}\leq C_1\|\nabla u\|_{L^2(\Omega\e)},&\quad  \forall u\in H^1(\Omega\e),\\	\label{lemma5:est2}
			\|\P\e u\|_{L^2(\Omega)}\leq C_2\| u\|_{H^1(\Omega\e)},&\quad \forall u\in H^1(\Omega\e).
		\end{align} 
	\end{lemma}
	
	\begin{proof}
		First of all, we observe that 
		\begin{gather}\label{Pie:prop}
			\P\ie u	= \P\ie (u-\la u\ra_{F\ie \setminus \overline{D\ie}}) +
			\la u\ra_{F\ie\setminus \overline{D\ie}}\text{\quad in }D\ie.
		\end{gather}
		since  the harmonic extension of a constant function $\la u\ra_{F\ie\setminus \overline{D\ie}}$ is the constant $\la u\ra_{F\ie\setminus \overline{D\ie}}$ itself.
		
		Using the coordinate transformation \eqref{coord} mapping
		$\B(2,0)$ and  $\D\ie$ to $F\ie$ and $D\ie$, respectively, we deduce from \eqref{assump:shape4}:
		\begin{multline}\label{lemma5:1}
			d\ie^{-2}\|\P\ie (u-\la u\ra_{F\ie\setminus \overline{D\ie}})\|^2_{L^2(F\ie)}+ 
			\|\nabla(\P\ie (u-\la u\ra_{F\ie\setminus \overline{D\ie}}))\|^2_{L^2(F\ie)}\\\leq 
			C \left(d\ie^{-2}\| u-\la u\ra_{F\ie\setminus \overline{D\ie}} \|^2_{L^2(F\ie\setminus\overline{D\ie})}+
			\|\nabla  u \|^2_{L^2(F\ie\setminus\overline{D\ie})}\right).
		\end{multline}
		One has the Poincar\'{e} inequality (see \eqref{Neumann:est:general}):
		\begin{align}\notag
			\| u-\la u\ra_{F\ie\setminus \overline{D\ie}}\|^2_{L^2(F\ie\setminus \overline{D\ie})}&\leq 
			(\Lambda_{\rm N}(F\ie\setminus \overline{D\ie}))^{-1}\|\nabla  u\|^2_{L^2(F\ie\setminus \overline{D\ie})}
			\\&= (\Lambda_{\rm N}(\mathbf{G}\ie))^{-1}d\ie^2\|\nabla  u\|^2_{L^2(F\ie\setminus \overline{D\ie})}.\label{lemma5:2}
		\end{align}
		Combining  \eqref{assump:shape3} with \eqref{lemma5:1}--\eqref{lemma5:2}, we get	
		\begin{gather}\label{lemma5:3}
			d\ie^{-2}\|\P\ie (u-\la u\ra_{F\ie\setminus \overline{D\ie}})\|^2_{L^2(F\ie)}+ 
			\|\nabla(\P\ie (u-\la u\ra_{F\ie\setminus \overline{D\ie}}))\|^2_{L^2(F\ie)}\leq 
			C  \|\nabla  u \|^2_{L^2(F\ie\setminus\overline{D\ie})} .
		\end{gather}	
		In particular, it follows from \eqref{lemma5:3} and \eqref{Pie:prop} that 
		\begin{align}\notag
			\sum_{i\in\I\e}\|\nabla(\P\e u)\|^2_{L^2( F\ie)}&=\sum_{i\in\I\e}	\|\nabla(\P\ie (u-\la u\ra_{F\ie\setminus \overline{D\ie}}))\|^2_{L^2(F\ie)}
			\\
			&\leq 
			C  \sum_{i\in\I\e}\|\nabla  u \|^2_{L^2( F\ie\setminus\overline{D\ie})}\leq 
			C\|\nabla u\|^2_{L^2(\Omega\e)},\label{lemma5:4}
		\end{align}	 
		From  \eqref{lemma5:4}, taking into account that 
		$\P\e u = u $ in $\Omega\e$, we conclude  \eqref{lemma5:est1}.
		Finally, by Lemma~\ref{lemma:MK06}  (applied with $D=D_1=F\ie$, 
		$D_2=F\ie\setminus\overline{B\ie}$) and \eqref{lemma5:4}, we get
		\begin{align*}
			\sum_{i\in\I\e}\|\P\e u\|^2_{L^2(F\ie)}&\leq 
			C\sum_{i\in\I\e}\left(\|\P\e u\|^2_{L^2(F\ie\setminus\overline{B\ie})} + 
			d\ie^2\|\nabla (\P\e  u)\|^2_{L^2(F\ie)}\right)
			\\&=
			C\sum_{i\in\I\e}\left(\| u\|^2_{L^2(F\ie\setminus\overline{B\ie})} + 
			d\ie^2\|\nabla (\P\e  u)\|^2_{L^2(F\ie)}\right)\leq
			C_1\|u\|^2_{H^1(\Omega\e)},
		\end{align*}
	whence, taking into account that 
	$\P\e u = u $ in $\Omega\e$, we conclude  the estimate \eqref{lemma5:est2}.
		The lemma is proven.
	\end{proof}
	
	\section{Proof of the main results\label{sec4}}
	
	\subsection{Preliminaries}
	
	First, we present two simple results that will be used repeatedly throughout the sequel.
	
	\begin{lemma}\label{lm:4:1}
		The following estimates hold:
		\begin{gather}\label{dQd}
			\exists C_1,\,C_2>0\ \forall\eps>0\ \forall i\in \I\e:\quad
			C_1 Q\ie \le d\ie^{n-1}r\ie^{-n} \leq C_2 Q\ie.
		\end{gather}
	\end{lemma}
	\begin{proof}
		The first estimate in \eqref{dQd} follows from  
		the bounds
		\begin{gather}\label{area:est}
			\area(\partial D\ie)=
			d\ie^{n-1}\area(\partial \D\ie)\le 
			C_{\rm tr}^2\, d\ie^{n-1}\vol(\D\ie)\le
			C_{\rm tr}^2\, d\ie^{n-1}\vol(\B(1,0)).
		\end{gather}
		Here we use the inequalities $\area(\partial \D\ie)\le 
		C_{\rm tr}^2\vol(\D\ie)$ (which is obtained by inserting $u\equiv 1$ into \eqref{assump:shape2})  and $\vol(\Gamma\ie)\ge Cr\ie^n$.
		The second estimate in \eqref{dQd} follows from  
		\eqref{Gamma:cond:2},
		\eqref{assump:shape1},
		\eqref{isoper}.  The lemma is proven.
	\end{proof}
	
	\begin{lemma}\label{lm:4:2}
		The following estimate holds:
		\begin{gather*}
			\exists C>0\ \forall\eps>0\ \forall v\in H^1_0(\Omega):\quad
			\sum_{i\in\I\e}Q\ie\|v\|^2_{L^2(\Gamma\ie)}\leq 
			C\|v\|^2_{H^1(\Omega)}.
		\end{gather*}
	\end{lemma}
	\begin{proof}
		Recall that   $\E :H^1_0(\Omega)\to H^1(\R^n)$ is the operator of extension by zero. 
		Let $\mu$ and $\nu$ are given in \eqref{mu} and \eqref{nu}.
		It follows from \eqref{assump:main} and \eqref{assump:Q2} that 
		\begin{gather}\label{Qsup}
			\exists C>0\ \forall k\in\Z^n\ \forall\eps>0:\quad \|Q\e\|_{L^\mu(\square_k)}\leq C.
		\end{gather}
		Then, using H\"older's inequality (recall that  $\mu^{-1}+2\nu^{-1}=1$), \eqref{Sobolev:nu} and \eqref{Qsup}, we obtain:
		\begin{align*}
			\sum_{i\in\I\e}Q\ie\|v\|^2_{L^2(\Gamma\ie)}&=
			\int_{\Omega}Q\e(x) |v(x)|^2\d x=
			\sum_{k\in\Z^n}\int_{ \square_k}Q\e(x) | \E v(x)|^2\d x
			\\
			&\leq
			\sum_{k\in\Z^n}\|Q\e\|_{L^\mu( \square_k)}\|\E v\|^2_{L^\nu(\square_k)}
			\leq
			C\sum_{k\in\Z^n}\|\E v\|^2_{H^1(\square_k)}
			=
			C\|v\|^2_{H^1(\Omega)}.
		\end{align*}
		The lemma is proven.
	\end{proof}

	\subsection{Proof of Theorem~\ref{th1}}
	
	First, we establish the estimates  \eqref{th1:norms}.
	Recall that 
	$\P\e:H^1(\Omega\e)\to H^1(\Omega)$ is an extension operator defined by \eqref{Pe}.
	One has:
	\begin{align*} 
		\suml_{i\in\I\e}\|u\|^2_{L^2(\partial D\ie)}&\le
		C\suml_{i\in\I\e}
		\left(Q\ie\|\P\e u\|^2_{L^2(R\ie\setminus \overline{B\ie})}+d\ie\zeta\ie\|\nabla(\P\e u)\|^2_{L^2(R\ie)}\right)
		\\
		& 
		\leq  
		C_1\|\P\e u\|^2_{H^1(\Omega)}\leq C_2\|u\|^2_{H^1(\Omega\e)},
	\end{align*}
	where on the first step we utilized Lemma~\ref{lemma1} and Lemma~\ref{lm:4:1},
	on the second step we used the boundedness of $d\ie\zeta\ie$ and 
	 Lemma~\ref{lm:4:2} (taking into account that $R\ie\setminus \overline{B\ie}\subset\Gamma\ie $), 
	and on the last step we applied Lemma~\ref{lemma5}.
	Thus, the right-hand-side estimate in \eqref{th1:norms} is proven.
	Also, by Lemmas~\ref{lemma2}, \ref{lemma4}, \ref{lemma5} (namely, the estimate \eqref{lemma5:est1}), Lemma~\ref{lm:4:1} and \eqref{assump:Q1},   we get
	\begin{align}
		\|u\|^2_{H^1(\Omega\e)}& \leq \|\P\e u\|^2_{H^1(\Omega)}=
		\|\nabla (\P\e u)\|^2_{L^2(\Omega)}+\sum_{i\in\I\e}\|\P\e u\|^2_{L^2(\Gamma\ie)}+
		\|\P\e u\|^2_{L^2(T\e)}
		\notag\\\notag
		&\leq
		\|\nabla (\P\e u)\|^2_{L^2(\Omega)}+
		C_1\sum_{i\in\I\e}Q\ie^{-1}\left(\|u\|^2_{L^2(\partial D\ie)}+
		d\ie\zeta\ie\|\nabla(\P\e u)\|^2_{L^2(\Gamma\ie)}\right)
		\\
		&+
		C_2\rho\e^2\|\P\e u\|^2_{H^1(\Omega)}
		\notag \leq C_3\left(\sum_{i\in\I\e}\|u\|^2_{L^2(\partial D\ie)}+\|\nabla (\P\e u)\|^2_{L^2(\Omega)}\right)+C_2\rho\e^2\|\P\e u\|^2_{H^1(\Omega)}\notag\\&\leq
		C_4\|u\|^2_{\HS\e}+C_2\rho\e^2\|\P\e u\|^2_{H^1(\Omega)}\leq
		C_4\|u\|^2_{\HS\e}+C_5\rho\e^2\| u\|^2_{H^1(\Omega\e)}.
		\label{Cpm:2}
	\end{align}
	Assuming $\eps$ to be sufficiently small in order to have $C_5  \rho\e^2\leq 1/2$ (recall that $\rho\e\to 0$),
	we conclude from \eqref{Cpm:2} the left-hand-side estimate in \eqref{th1:norms}.
	\smallskip
	
	Now, we proceed to the proof of the main estimate \eqref{th1:est}.
	Let $f\in \HS=H^1_0(\Omega)$. We set  $u\e\ceq\Res\e \J\e f$ (recall that $\J\e f\e=f\restr_{\Omega\e}$) and $u\ceq\Res f$.

	The following equality follows easily from the definitions of $\Res\e$ and $\Res$
	(see~\eqref{variat:eps} and \eqref{variat}) and the fact that 
	$\P\e u=u$ on $\Omega\e$:
	\begin{gather} \label{I123}
		(\Res\e\J\e f-\J\e\Res f,v)_{\HS\e}=(u\e-\J\e u,v)_{\HS\e}=I_{1,\eps}+I_{2,\eps}+I_{3,\eps},\quad \forall v\in\HS\e,
	\end{gather}
	where
	\begin{align*}
		I_{1,\eps}&\ceq\sum_{i\in\I\e}(f,v)_{L^2(\partial D\ie)}-
		\int_\Omega Q(x)f(x)\overline{\P\e v(x)}\d x,\\
		I_{2,\eps}&\ceq
		-\sum_{i\in\I\e}(u,v)_{L^2(\partial D\ie)}+\int_\Omega Q(x)u(x)\overline{\P\e v(x)}\d x,\\
		I_{3,\eps}&\ceq \sum_{i\in\I\e}(\nabla u,\nabla(\P\e v))_{L^2(D\ie)}.
	\end{align*}
	
	In the following three lemmas we estimate each of the terms $I_{k,\eps}$, $k=1,2,3$.

	\begin{lemma}\label{lemma:I1}
		One has the estimate
		\begin{gather}\label{I1:est}
			|I_{1,\eps}|\leq  
			C{\,\max\left\{\kappa\e,\,  (\sup_{i\in\I\e}d\ie\zeta\ie)^{1/2}\right\}}\|f\|_{\HS}\|v\|_{\HS\e}.
		\end{gather}
	\end{lemma}
	
	\begin{proof} 
		We have
		$$
		I_{1,\eps}=I_{11,\eps}+I_{12,\eps},
		$$
		where
		\begin{gather*}
			I_{11,\eps}= 
			\sum_{i\in\I\e}
			 Q\ie (f ,\P\e v)_{L^2(\Gamma\ie)} -
			\int_\Omega Q(x)f(x)\overline{\P\e v(x)}\d x,\\ 
			I_{12,\eps}=\sum_{i\in\I\e}\left( (f,v)_{L^2(\partial D\ie)}- {Q}\ie(f,\P\e v)_{L^2(\Gamma\ie)}\right).
		\end{gather*}
		Note that the sum 	$\sum_{i\in\I\e}
		Q\ie  (f ,\P\e v)_{L^2(\Gamma\ie)}$ above is indeed finite: by the Cauchy--Schwarz inequality and Lemma~\ref{lm:4:2}, we obtain
		$$
		\Big|\sum_{i\in\I\e}
		Q\ie (f ,\P\e v)_{L^2(\Gamma\ie)}\Big|^2\leq 
		 \sum_{i\in\I\e}Q\ie\|f\|^2_{L^2(\Gamma\ie)} 
	     \sum_{i\in\I\e}Q\ie\|\P\e v\|^2_{L^2(\Gamma\ie)}
	     \leq 
		C\|f\|_{\HS}\|\P\e v\|_{\HS}.
		$$
		
		Recall that $\mu $ and $\nu$ are defined in \eqref{mu} and \eqref{nu}, respectively. 
		By successively applying Hölder's inequality,  \eqref{assump:main}, \eqref{Sobolev:nu}, the Cauchy-Schwarz inequality, 
		Lemma~\ref{lemma5},  and, finally,
		\eqref{norm:equivs} and \eqref{th1:norms},
		we get
		\begin{align}\notag
			|I_{11,\eps}|&=
			\left| 
			\int_{\Omega}(Q\e(x)-Q(x))f(x)\overline{\P\e v(x)}\d x\right|=
			\left| 
			\sum_{k\in\Z^n}\int_{\Omega\cap\square_k}(Q\e(x)-Q(x))  f(x)\overline{  \P\e v(x)}\d x\right|
			\\\notag &\le
			\sum_{k\in\Z^n}\|Q\e-Q\|_{L^\mu(  \square_k)}\|\mathscr{E}f\|_{L^\nu( \square_k)}\|\mathscr{E}\P\e v\|_{L^\nu( \square_k)}
			\\\notag
			&\leq 
			\kappa\e\sum_{k\in\Z^n}\|\E  f\|_{L^\nu(\square_k)}\|\E \P\e v\|_{L^\nu(\square_k)} \leq
			C_1\kappa\e\sum_{k\in\Z^n}\|\E  f\|_{H^1(\square_k)}\|\E \P\e v\|_{H^1(\square_k)}
			\\\notag &\le C_1\kappa\e\Big(\sum_{k\in\Z^n}\|\E  f\|^2_{H^1(\square_k)}\Big)^{1/2}\Big(\sum_{k\in\Z^n}\|\E \P\e v\|^2_{H^1(\square_k)}\Big)^{1/2} = C_1\kappa\e\|\E  f\|_{H^1(\R^n)}\|\E  \P\e v\|_{H^1(\R^n)}\\&
			= C_1\kappa\e\|f\|_{H^1(\Omega)}\|\P\e v\|_{H^1(\Omega)}\leq C_2\kappa\e\|f\|_{H^1(\Omega)}\|v\|_{H^1(\Omega\e)}\leq C_3{\kappa\e}\|f\|_{\HS}\|v\|_{\HS\e}.\label{I11:final}
		\end{align}

		Now, we proceed to the term $I_{12,\eps}$. One has:
		\begin{align*}
			I_{12,\eps}&=\sum_{i\in\I\e}  (f-\la f\ra_{\Gamma\ie},v)_{L^2(\partial D\ie)}
			+
			\sum_{i\in\I\e}\la f\ra_{\Gamma\ie}\left(\la v\ra_{\partial D\ie}-\la \P\e v\ra_{  \Gamma\ie}\right)\area(D\ie)
			\\
			&-\sum_{i\in\I\e}(Q\ie (f- \la f\ra_{\Gamma\ie}),\P\e v)_{L^2(\Gamma\ie)}.
		\end{align*}
		Note that all terms in the above equality are indeed finite. For example, we have
		\begin{align*}
		\Big|\sum_{i\in\I\e}\la f\ra_{\Gamma\ie} \la v\ra_{\partial D\ie}\area(D\ie)\Big|^2
		&\leq 
		 \sum_{i\in\I\e}|\la f\ra_{\Gamma\ie}|^2\area(D\ie) \cdot 
		 \sum_{i\in\I\e}|\la v\ra_{\partial D\ie}|^2\area(D\ie) 
		 \\
		 &\le \sum_{i\in\I\e}Q\ie\|f\|^2_{L^2(\Gamma\ie)} \cdot 
		 \sum_{i\in\I\e}\| v\|^2_{L^2(\partial D\ie)}\leq 
		 C\|f\|^2_{\HS} \|v\|^2_{\HS\e},
		\end{align*}
		where we used Lemma~\ref{lm:4:2} and \eqref{norm:equivs}; the remaining terms can be estimated in a similar way.
		Using the Cauchy-Schwarz inequality, we get
		$
		|I_{12,\eps}|\leq J_{1,\eps}+J_{2,\eps}+J_{3,\eps}
		$
		with
		\begin{align*}
			J_{1,\eps}= &
			\sqrt{\sum_{i\in\I\e}\|f-\la f\ra_{\Gamma\ie}\|^2_{L^2(\partial D\ie)}\sum_{i\in\I\e}\|v\|^2_{L^2(\partial D\ie)}},\\
			J_{2,\eps}= &
			\sqrt{
				\sum_{i\in\I\e}{\area(\partial D\ie) }|\la f\ra_{\Gamma\ie}|^2
				\sum_{i\in\I\e}\area(\partial D\ie)|\la v\ra_{\partial D\ie}-\la\P\e v\ra_{\Gamma\ie}|^2},\\
			J_{3,\eps}= & \sqrt{\sum_{i\in\I\e}Q\ie\|f-\la f\ra_{\Gamma\ie}\|^2_{L^2(\Gamma\ie)}
				\sum_{i\in\I\e}Q\ie\|\P\e v\|^2_{L^2(\Gamma\ie)}}.
		\end{align*}
		
		Employing the Poincar\'e inequality \eqref{Neumann:est:general} with $D=\Gamma\ie$ and taking into account \eqref{LambdaN:Gamma:est} and \eqref{Gamma:cond:2}, we get
		\begin{gather}\label{Poincare:Gamma}
			\|f-\la f\ra_{\Gamma\ie}\|_{L^2(\Gamma\ie)}^2\leq 
			C(\Lambda_{\rm N}(\Gamma\ie))^{-1}\|\nabla f\|_{L^2(\Gamma\ie)}^2
			\le C_1r\e^2\|\nabla f\|_{L^2(\Gamma\ie)}^2.
		\end{gather} 
		Then, using Lemma~\ref{lemma1}, \eqref{Poincare:Gamma} and  \eqref{assump:D1+}, we obtain
		\begin{align}\notag
			\|f-\la f\ra_{\Gamma\ie}\|^2_{L^2(\partial D\ie)}
			&\leq
			C\left(d\ie^{n-1} r\ie^{-n}\|f-\la f\ra_{\Gamma\ie}\|^2_{L^2(\Gamma\ie)}
			+d\ie\zeta\ie\|\nabla f\|^2_{L^2(\Gamma\ie)}\right)
			\\\notag
			&\leq 
			C_1\max\{d\ie^{n-1}r\ie^{2-n},\,d\ie\zeta\ie\}\|\nabla f\|^2_{L^2(\Gamma\ie)}\le
			C_2 d\ie\zeta\ie\|\nabla f\|^2_{L^2(\Gamma\ie)},
		\end{align}
		whence, we infer
		\begin{gather}\label{Q1:est}
			J_{1,\eps}\le  
			C {( \sup_{i\in\I\e} d\ie\zeta\ie)^{1/2}}\|f\|_{\HS}\|v\|_{\HS\e}.
		\end{gather} 
		
		Next, one has
		\begin{align}\notag
			J_{2,\eps}&\le C\Big(\sum_{i\in\I\e}{Q}\ie\|f\|^2_{L^2(\Gamma\ie)}\Big)^{1/2}
			\Big(\sum_{i\in\I\e}d\ie\zeta\ie\|\nabla(\P\e v)\|_{L^2(\Gamma\ie)}^2\Big)^{1/2}\\
			&\leq C(\sup_{i\in\I\e}d\ie\zeta\ie)^{1/2}\|f\|_{H^1(\Omega)}\|\nabla(\P\e v)\|_{L^2(\Omega)}\notag
			\\\label{Q2:est}
			&\leq C_1(\sup_{i\in\I\e}d\ie\zeta\ie)^{1/2}\|f\|_{H^1(\Omega)}\|\nabla v\|_{L^2(\Omega\e)}\leq C_1{ (\sup_{i\in\I\e}d\ie\zeta\ie)^{1/2}}\|f\|_{\HS}\|  v\|_{\HS\e},
		\end{align}
		where on the first step we used the Cauchy-Schwarz inequality, 
	 the estimate \eqref{area:est} and Lemma~\ref{lemma3},
	 on the second step we used Lemma~\ref{lm:4:2},  
	 on the third step we applied Lemma~\ref{lemma5}, and on the last step  we utilized \eqref{norm:equivs}.
		
		Finally, 
		we get
		\begin{align}
			\notag
			J_{3,\eps}&\leq C\big(\sup_{i\in\I\e}Q\ie r\ie^2\big)^{1/2}\|\nabla f\|_{L^2(\Omega)} \|\P\e v\|_{H^1(\Omega)}
			\le  C_1\big(\sup_{i\in\I\e}d\ie^{n-1} r\ie^{2-n}\big)^{1/2}\|\nabla f\|_{L^2(\Omega)} \|\P\e v\|_{H^1(\Omega)}
			\\
			&\leq C_1\big(\sup_{i\in\I\e}d\ie^{n-1} r\ie^{2-n}\big)^{1/2}\|\nabla f\|_{L^2(\Omega)} \|  v\|_{H^1(\Omega\e)}\le C_3{ \big(\sup_{i\in\I\e}d\ie \big)^{1/2}}\|f\|_{\HS} \|v\|_{ \HS\e},\label{Q3:est}
		\end{align}
		where on the first step we used \eqref{Poincare:Gamma} and Lemma~\ref{lm:4:2}, on the second step we used Lemma~\ref{lm:4:1},
		on the third step we applied Lemma~\ref{lemma5},
		and on the last step we utilized 
		  \eqref{assump:D1+},  \eqref{th1:norms}.
		
		Combining \eqref{I11:final},  \eqref{Q1:est}, \eqref{Q2:est}, \eqref{Q3:est}, we arrive at the 
		desired estimate \eqref{I1:est}.
	\end{proof}
	
	\begin{lemma}\label{lemma:I2}
		One has the estimate
		\begin{gather}\label{I2:est}
			|I_{2,\eps}|\leq  C{\,\max\left\{\kappa\e,\,  (\sup_{i\in\I\e}d\ie\zeta\ie)^{1/2}\right\}} 
			\|f\|_{\HS}\|v\|_{\HS\e}.
		\end{gather}
	\end{lemma}
	
	\begin{proof}
		Repeating verbatim the proof of Lemma~\ref{lemma:I1} we get 
		\begin{gather}\label{I2:step1}
			|I_{2,\eps}|\leq  C\max\left\{\kappa\e,\,  (\sup_{i\in\I\e}d\ie\zeta\ie)^{1/2}\right\}
			\| u\|_{\HS}\|v\|_{\HS\e}.
		\end{gather}	
		Furthermore, one has the bound 
		\begin{gather}\label{u:bound}
		\|u\|_{\HS}\leq \|f\|_{\HS},
		\end{gather}
		which is obtained  by testing \eqref{variat} with $v=u$; combined with \eqref{I2:step1}, this yields the desired estimate \eqref{I2:est}.
	\end{proof}
	
	\begin{lemma}\label{lemma:I3}
		One has the estimate
		\begin{gather}\label{I3:est}
			|I_{3,\eps}|\leq  C{  \big(\sup_{i\in\I\e}d\ie\big)^{1/2}}\|f\|_{\HS}\|  v\|_{\HS\e}.
		\end{gather}

	\end{lemma}
	\begin{proof}
	First, we observe that the assumptions we imposed on the function $Q$ imply that the product 
	$Qf$ belongs to $L^2(\Omega)$ for any $f\in\HS$.
	Indeed, using H\"older's inequality, \eqref{assump:Q2}, \eqref{Sobolev:nu}, \eqref{norm:equivs},
	we conclude ($\mu$ and $\nu$ below are given in \eqref{mu} and \eqref{nu}, respectively, and one has $(2\mu)^{-1}+\nu^{-1}=1/2$):
	\begin{align}\notag
		\|Q f\|^2_{L^2(\Omega)}&=
		\sum_{k\in \Z^n}\|Q\E f\|^2_{L^2( \square_k)}\le
		\sum_{k\in\Z^n}
		\|Q\|^2_{L^{2\mu}(\square_k)}
		\|\E f\|^2_{L^\nu(\square_k)}
		\\
		&\leq C\sum_{k\in\Z^n}\|\E f\|^2_{H^1(\square_k)}=
		C\|f\|^2_{H^1(\Omega)}\leq C_1\|f\|^2_{\HS}.\label{QL2}
	\end{align}
	From \eqref{QL2} and the uniform regularity of $\Omega$, we conclude   \cite[Theorem~2]{Br61}   that   $u=\Res f$ belongs to $H^2(\Omega)$, with the  estimate
	\begin{gather}\label{el:reg}
		\|u\|_{H^2(\Omega)}\leq C\|Q f\|_{L^2(\Omega)}\leq   C_1\|f\|_{\H}.
	\end{gather}

	Now, we turn to the proof of the estimate \eqref{I3:est}.
		One has
		\begin{align*}
			|I_{3,\eps}|^2&\leq 
			C\sum_{i\in\I\e}\sum_{j=1}^n d\ie
			\left(  Q\ie\|\partial_j u\|^2_{L^2(R\ie)}+
			d\ie \zeta\ie\|\nabla (\partial_j u)\|^2_{L^2(R\ie\setminus\overline{B\ie})}\right) 
			\sum_{i\in\I\e}\|\nabla(\P\e v)\|^2_{L^2(D\ie)}
			\\
			&\leq 
			C\sup_{i\in\I\e}d\ie
			\sum_{j=1}^n \left(  \|\partial_j u\|^2_{H^1(\Omega)}+\|\nabla(\partial_j u)\|^2_{L^2(\Omega)}\right) 
			\|\nabla(\P\e v)\|^2_{L^2(\Omega)}
			\\&\leq 
			C_1\sup_{i\in\I\e}d\ie\,\|u\|_{H^2(\Omega)}^2\|\nabla v\|_{L^2(\Omega\e)}^2 \leq
			C_2\sup_{i\in\I\e} d\ie\|f\|_{\HS}^2\| v\|_{\HS\e}^2,
		\end{align*}
		where on the first step we used Lemmas~\ref{lemma1+} and 
		 \ref{lm:4:1}, on the second step we applied Lemma~\ref{lm:4:2}, on the third step we utilized Lemma~\ref{lemma5}, and on the last step we used  \eqref{el:reg}.
		The lemma is proven.
	\end{proof}
	
	Combining \eqref{I123}, \eqref{I1:est}, \eqref{I2:est}, \eqref{I3:est},
	we arrive at 
	\begin{gather*}
		\forall v\in \HS\e:\quad \left|(\Res\e \J\e f  - \J\e\Res f,v)_{\H\e}\right|\leq C{\,\max\left\{\kappa\e,\,  (\sup_{i\in\I\e}d\ie\zeta\ie)^{1/2}\right\}}\|f\|_{\HS}\|v\|_{\HS\e},
	\end{gather*}
	which implies the desired estimate \eqref{th1:est}. Theorem~\ref{th1} is proven.
	
	\subsection{Proof of Theorem~\ref{th2}}
	
	One has the bound (see \eqref{u:bound})
	\begin{gather}\label{Rff}
		\|\Res f\|_{\HS}\leq \|f\|_{\HS},
	\end{gather}moreover, 
	$
	(\Res f,f)_{\HS}=({Q} f,f)_{L^2(\Omega)}\ge 0.
	$
	These two inequalities imply $\sigma(\Res)\subset [0,1]$.  
	Similarly, $\sigma(\Res\e)\subset [0,1]$.
	It is easy to see that   $0\in\sigma(\Res\e)$.  
	
In the proofs of the lemmas below, we will use the   notation
 $$d\e\ceq\sup_{i\in\I\e}d\ie.$$

	\begin{lemma}\label{lemma:Haus1}
		One has the following estimate:
		\begin{gather}\label{distH:est1}
			\forall z\in\sigma(\Res\e):\quad \dist(z,\sigma(\Res))\leq 
			C\max\left\{\kappa\e,\,  (\sup_{i\in\I\e}d\ie\zeta\ie)^{1/2}\right\}, 
		\end{gather}
		where the constant $C>0$ is independent of $\eps$ and $z$. 
	\end{lemma}
	
	\begin{proof}
		One has the inequality 
		\begin{gather}\label{Haus1a}
			\forall z\in\C\quad \forall\psi\in\HS\setminus\{0\} \colon \quad 
			\dist(z,\sigma(\Res))
			\leq \frac{\|(\Res -z\Id)\psi\|_{\HS}}{\|\psi\|_{\HS}}
		\end{gather}
		(for $z\in \sigma(\Res)$  this inequality is trivial, while for 
		$z\in\C\setminus\sigma(\Res)$ it follows easily from the  equality 
		$\|(\Res -z\Id)^{-1}\|_{\HS}=(\dist(z,\,\sigma(\Res))^{-1}$).

		Let us fix $z\in\sigma(\Res\e) $; then
		for each $\delta>0$ there exists 
		$\psi\ed\in \HS\e$ such that 
		\begin{gather*}
			\|\Res\e \psi\ed - z\psi\ed\|_{\HS\e}\le \delta,\quad \|\psi\ed\|_{\HS\e}=1.
		\end{gather*}
		Recall that  $\P\e:H^1(\Omega\e)\to H^1(\Omega)$ is given   
		by \eqref{Pe}. 
		From \eqref{norm:equivs}, \eqref{th1:norms} and Lemma~\ref{lemma5} we infer
		\begin{gather}\label{P:est1}
			\exists C_1,C_2>0\ \forall u\in \HS\e:\quad
			C_1\|\P\e u\|_{\HS}\leq \|u\|_{\HS\e}\leq C_2\|\P\e u\|_{\HS}.
		\end{gather}
		(in the second estimate we utilize the fact that $\P\e u=u$ on $\Omega\e$).
		Using \eqref{Haus1a}--\eqref{P:est1}, we get
		\begin{align} \notag
			&\dist(z,\,\sigma(\Res)) 
			\leq  
			\frac{\|(\Res-z\Id) \P\e \psi\ed\|_{\HS}}
			{\| \P\e \psi\ed\|_{\HS}}\\\notag
			&\leq 
			\frac{\|\Res\P\e\psi\ed - \P\e   \Res\e\psi\ed\|_{\HS}+
				\| \P\e (\Res\e - z\Id)\psi\ed\|_{\HS}}
			{\| \P\e \psi\ed\|_{\HS}} 
			\leq
			C\left(\|\Res\P\e \psi\ed - \P\e   \Res\e\psi\ed\|_{\HS}+\delta\right)\\ \notag
			&\leq  C\left(
			\|\Res\P\e \psi\ed-\P\e\J\e\Res\P\e \psi\ed\|_{\HS}+\|\P\e\J\e\Res\P\e\psi\ed  -   \P\e \Res\e\psi\ed\|_{\HS}+\delta\right)
			\\\label{Q1Q2}
			&=
			C\big(
			\underset{J_{1,\eps}\ceq}{\underbrace{\|\Res\P\e \psi\ed-\P\e\J\e\Res\P\e \psi\ed\|_{\HS}}}+
			\underset{J_{2,\eps}\ceq}{\underbrace{\|\P\e\J\e\Res\P\e\psi\ed  -  \P\e  \Res\e\J\e\P\e\psi\ed\|_{\HS}}}+\delta\big).
		\end{align}
		where $C$ is a constant, which is independent of $\eps$, $\delta$ and $z$.
		
		To estimate the term $J_{1,\eps}$, we first observe that  
		one has the bound
		\begin{multline}\label{aux:ext}
			\exists C>0\ \forall\eps>0\ \forall i\in\I\e:\\ 
			\|\Res f - \P\e \J\e  \Res f\|_{H^1(D\ie)}^2  \leq C d\ie^2\left(\| Qf\|^2_{L^2(D\ie)} + \|Q\Res f\|_{L^2(D\ie)}^2\right),\ \ \forall f\in\HS.
		\end{multline}
		(recall that the assumptions imposed on $Q$ ensure that  
		$Qf\in L^2(\Omega)$ provided $f\in \HS$, cf.~\eqref{QL2}).
		Indeed, the function $\Res f$ satisfies 
		$-\Delta \Res f + {Q} \Res f = {Q} f$ in $D\ie$, while the function
		$\P\e\J\e \Res f$ is harmonic in $D\ie$ (see the definition of   $\P\e$), therefore,   
		$w\e\ceq \Res f - \P\e\J\e \Res f$ satisfies
		\begin{gather}\label{w1}
			-\Delta w\e=  {Q}(f-\Res f)\text{ in }D\ie.
		\end{gather}
		Furthermore,  $\Res f$ and $\P\e\J\e \Res f$ coincide on $\partial D\ie$
		(since $\P\e u= u $ on $\Omega\e$),
		whence 
		\begin{gather}\label{w2}
			w\e=0\text{ on }\partial D\ie.
		\end{gather}
		The solution $w\e$ to \eqref{w1}--\eqref{w2} satisfies the following standard estimates:
		\begin{align}\label{w3}
			\|w\e\|_{L^2(D\ie)}^2&\leq  ( \Lambda_{\rm D}(D\ie)) ^{-2}\| {Q}(f-\Res f)\|^2_{L^2(D\ie)},
			\\ 
			\|\nabla w\e\|_{L^2(D\ie)}^2&\leq  (\Lambda_{\rm D}(D\ie))^{-1} \| {Q}(f-\Res f)\|^2_{L^2(D\ie)}.
		\end{align}
		Hereinafter, by $ \Lambda_{\rm D}(D)$ we denote the smallest eigenvalue of the Dirichlet Laplacian on a bounded domain $D$. 
		Evidently, one has
		\begin{gather}\label{LambdaD}
			\Lambda_{\rm D}(D\ie)=d\ie^{-2}  \Lambda_{\rm D}(\D\ie)\geq d\ie^{-2}\Lambda_{\rm D}(\B(1,0))
		\end{gather}		
		(the latter inequality follows from the domain monotonicity of  Dirichlet eigenvalues).	From \eqref{w3}--\eqref{LambdaD} we conclude the bound \eqref{aux:ext}.
		
		Taking into account that 
		$$\Res\P\e \psi\ed-\P\e\J\e\Res\P\e \psi\ed =0\text{ in }\Omega\e,$$ 
		and using \eqref{aux:ext}, \eqref{QL2}, \eqref{Rff}, \eqref{P:est1},
		and $\|\psi_{\eps,\delta}\|_{\HS\e}=1$,   
			we obtain
		\begin{align}\notag
			J_{1,\eps}^2
			&\le C \sum_{i\in\I\e} d\ie^2\Big(\|Q\P\e\psi\ed\|^2_{L^2(D\ie)}+ \|Q\Res \P\e\psi\ed\|^2_{L^2(D\ie)} \Big)\notag\\\notag & 
			\leq
			 Cd\e^2\Big(\|Q\P\e\psi\ed\|^2_{L^2(\Omega)}+ \|Q\Res \P\e\psi\ed\|^2_{L^2(\Omega)} \Big)\\&\leq C_1d\e^2\Big(\|\P\e\psi\ed\|^2_{\HS}+\|\Res\P\e\psi\ed\|^2_{\HS} \Big)\leq C_2d\e^2 \|\P\e\psi\ed\|^2_{\HS}\leq
			C_3d\e^2  .\label{QQ1}
		\end{align} 
		
		To estimate the term $J_{2,\eps}$ we use \eqref{P:est1}, Theorem~\ref{th1}, and $\|\psi_{\eps,\delta}\|_{\HS\e}=1$:
		\begin{align}\notag
			J_{2,\eps}\leq C\| \J\e\Res\P\e\psi\ed  -     \Res\e\J\e\P\e\psi\ed\|_{\HS\e}
			&\leq C_1{\,\max\left\{\kappa\e,\,  (\sup_{i\in\I\e}d\ie\zeta\ie)^{1/2}\right\}}\|\P\e\psi\ed\|_{\HS}\\\label{QQ2}
			&\leq C_2{\,\max\left\{\kappa\e,\,  (\sup_{i\in\I\e}d\ie\zeta\ie)^{1/2}\right\}}.
		\end{align}
		
		Combining \eqref{Q1Q2}, \eqref{QQ1}, \eqref{QQ2}, we obtain
		\begin{align}\notag
			\forall z\in \sigma(\Res\e)\ \forall\delta>0 :\quad    \dist(z,\,\sigma(\Res)) 
			\leq C\left({\max\left\{\kappa\e,\,  (\sup_{i\in\I\e}d\ie\zeta\ie)^{1/2}\right\}}
			+\delta\right).  
		\end{align}
		Passing to the limit $\delta\to 0$ in the above estimate, we get \eqref{distH:est1}.  The lemma is proven.
	\end{proof}    
	
	\begin{lemma}\label{lemma:Haus2}
		One has the following estimate:
		\begin{gather}\label{distH:est2}
			\forall z\in\sigma(\Res ):\quad \dist(z,\sigma(\Res\e))\leq 
			C{\,\max\left\{\kappa\e,\,  (\sup_{i\in\I\e}d\ie\zeta\ie)^{1/2}\right\}}, 
		\end{gather}
		where the constant $C>0$ is independent of $\eps$ and $z$. 
	\end{lemma}
	
	\begin{proof}
		Similarly to \eqref{Haus1a}, one has
		\begin{gather}\label{Haus1}
			\forall z\in\C\quad \forall\psi\in\HS\e\setminus\{0\} \colon \quad 
			\dist(z,\sigma(\Res\e ))
			\leq \frac{\|(\Res\e -z\Id)\psi\|_{\HS\e}}{\|\psi\|_{\HS\e}}.
		\end{gather}
		
		Recall that $\J\e:\HS\to\HS\e$ is defined by $\J\e f = f\restr_{\Omega\e}$.
		From \eqref{norm:equivs}, \eqref{th1:norms} we infer
		\begin{gather}\label{Je:norm}
			\|\J\e f\|_{\HS\e}\leq 
			C\|\J\e f\|_{H^1(\Omega\e)}\leq 
			C\|f\|_{H^1(\Omega)}\leq C_1\|f\|_{\HS},\ 
			\forall f\in\HS.
		\end{gather}
		
		Let us now \emph{fix }
		\begin{gather}\label{C'}
			z\in\sigma(\Res)\cap [C'd\e^{1/2},1],
		\end{gather}
		where  the constant $  C'>0$  will be specified later on.
		Since $z$ belongs to the spectrum of $\Res$,  
		for each $\delta>0$ there exists 
		$\psi_\delta\in \HS$ such that 
		\begin{gather}\label{Haus2}
			\|\Res \psi_\delta - z\psi_\delta\|_{\HS}<\delta,\quad \|\psi_\delta\|_{\HS}=1.
		\end{gather}
		In the following we consider only $\delta$   satisfying
		\begin{gather*}
			\delta<z/2.
		\end{gather*}
		
		Using \eqref{Haus1}--\eqref{Haus2} and   Theorem~\ref{th1}, we obtain
		\begin{align}\notag
			\dist(z,\,\sigma(\Res\e)) 
			&\leq  
			\frac{\|(\Res\e-z\Id) \J\e \psi_\delta\|_{\HS\e}}
			{\| \J\e \psi_\delta\|_{\HS\e}} \leq 
			\frac{\|(\Res\e\J\e - \J\e   \Res)\psi_\delta\|_{\HS\e}+
				\| \J\e (\Res - z\Id)\psi_\delta\|_{\HS\e}}
			{\| \J\e \psi_\delta\|_{\HS\e}} \\
			\label{Haus6}
			&\leq
			C_2 \,\frac{{ \max\left\{\kappa\e,\,  (\sup_{i\in\I\e}d\ie\zeta\ie)^{1/2}\right\}}+\delta}{\| \J\e \psi_\delta\|_{\HS\e}},
		\end{align}
		where $C_2$ is a constant, which is independent of $\eps$ and $\delta$.
		
		It remains  to estimate the norms  $\| \J\e \psi_\delta\|_{\HS\e}$ from below.
		We denote
		$g_\delta\ceq\Res\psi_\delta-z\psi_\delta.$
		Using \eqref{Je:norm} and \eqref{Haus2} we get
		\begin{align}\notag
			\| \J\e \psi_\delta\|_{\HS\e}=
			z^{-1}\| \J\e \Res\psi_\delta-\J\e g_\delta\|_{\HS\e}
			&\geq z^{-1}\left(\| \J\e \Res\psi_\delta \|_{\HS\e}-
			\| \J\e g_\delta\|_{\HS\e}\right)\\
			&\geq z^{-1}\left(\| \J\e \Res\psi_\delta \|_{\HS\e}-
			C_1\delta\right)  , 
			\label{Haus3}
		\end{align}
		where  $C_1$ is a constant from \eqref{Je:norm}.
		We have
		\begin{align}\notag \| \J\e \Res\psi_\delta \|_{\HS\e}^2&\ge
			C_3\|\J\e \Res\psi_\delta\|^2_{H^1(\Omega\e)}
			=C_3\| \Res\psi_\delta\|^2_{H^1(\Omega)}-C_3\sum_{i\in\I\e}\| \Res\psi_\delta\|^2_{H^1(D\ie)}
			\\\notag
			&\geq C_4\| \Res\psi_\delta\|^2_{\HS}-C_3\sum_{i\in\I\e}\|\Res\psi_\delta\|^2_{H^1(D\ie)}
			=C_4\|z\psi_\delta + g_\delta\|^2_{\HS}-C_3\sum_{i\in\I\e}\| \Res\psi_\delta\|^2_{H^1(D\ie)}\\ 
			&\geq C_4(z-\|g_\delta\|_{\HS})^2-C_5  d\e\|\Res\psi_\delta\|^2_{H^2(\Omega)}
			\geq C_4z^2/4-C_6d\e,
			\label{Haus4}
		\end{align}  
		where 
	to get the first inequality  we used \eqref{th1:norms},
	to get the second one we utilized \eqref{norm:equivs},
	to get the third one we used the  estimate
	\begin{align*}
		\sum_{i\in\I\e}\|u\|^2_{L^2(D\ie)} 
		&\leq C_0\sum_{i\in\I\e}d\ie\left(Q\ie\| u\|^2_{L^2(R\ie\setminus\overline{B\ie})}+d\ie\zeta\ie\|\nabla u\|^2_{L^2(R\ie)}\right) 
		\leq 
		C d\e \|u\|^2_{H^1(\Omega)},\ \forall u\in H^1(\Omega)
	\end{align*}
	(it follows from Lemmas~\ref{lemma1+}, \ref{lm:4:1}, \ref{lm:4:2}),
	and on the last step we used \eqref{el:reg}, $\|\psi_\delta\|_{\HS}=1$ and the bound
		$\|g_\delta\|_{\HS}<\delta<z/2$. 
		Note that the constant $C_4$ and $C_6$ standing in the right-hand-side of \eqref{Haus4} are independent of $z$ and $\delta$.
		Now, we choose the constant $ C'$ from \eqref{C'} by $$C'\ceq 2\sqrt{(C_6+1)/C_4}.$$ With this choice of $C'$ the right-hand-side of \eqref{Haus4}
		is strictly positive -- it is bounded from below by $d\e$. Then, combining \eqref{Haus3} and \eqref{Haus4}, and taking into account \eqref{C'}, we obtain 
		\begin{align}\label{Haus5}
			\|\J\e\psi_\delta\|_{\HS\e}&\ge 
			\sqrt{ C_4/4-C_6 d\e z^{-2}}-C_1\delta z^{-1}\geq
			C_7-C_1\delta z^{-1}, 
		\end{align}
		where the independent on $z$ and $\delta$ constant $C_7>0$ is given by $$C_7\ceq \sqrt{C_4/4- C_6/(C')^2}=\sqrt{C_4(C_6+1)^{-1}}/2.$$
		In the following, we choose 
		\begin{gather*} 
			\delta<C_7 z/C_1.
		\end{gather*}
		With this restriction on $\delta$ the right-hand side of  \eqref{Haus5} is positive.
		Then, combining \eqref{Haus6} and \eqref{Haus5},
		we arrive at the estimate
		\begin{multline*}
			\forall  z\in\sigma(\Res)\cap [ C'd\e^{1/2},1]\ \forall \delta\in (0,\min\{z/2,
			{C_7 z/C_1}\} ]:\\    \dist(z,\,\sigma(\Res\e)) 
			\leq
			C_2\,\frac{{\max\left\{\kappa\e,\,  (\sup_{i\in\I\e}d\ie\zeta\ie)^{1/2}\right\}}+\delta}{C_7-C_1\delta z^{-1}},
		\end{multline*}
		whence, passing to the limit $\delta\to 0$, we get
		\begin{align}\label{zin1}
			\forall  z\in\sigma(\Res)\cap [C'd\e^{1/2},1]:\quad
			\dist(z,\,\sigma(\Res\e)) 
			\leq
			C\,{\max\left\{\kappa\e,\,  (\sup_{i\in\I\e}d\ie\zeta\ie)^{1/2}\right\}}.
		\end{align}
		Since $0\in\sigma(\Res\e)$, we  have the trivial property
		\begin{align}\label{zin2}
			\forall  z\in\sigma(\Res)\cap [0,C'd\e^{1/2}):\quad
			\dist(z,\,\sigma(\Res\e)) 
			\leq
			C'd\e^{1/2}.
		\end{align}
		From \eqref{zin1}--\eqref{zin2} we conclude the desired bound \eqref{distH:est2}.
		The lemma is proven.
	\end{proof}
	
	The desired estimate \eqref{th2:est} follows immediately from \eqref{distH:est1}, \eqref{distH:est2} and the definition of the Hausdorff distance.
	Theorem~\ref{th2} is proven.

	\section*{Acknowledgements}
	The authors are grateful to Jean Lagacé for  fruitful discussions on this problem.
	The first author is grateful to Excellence Project FoS UHK 2204/2025-2026 for the financial support. This research was initiated during the second author’s visit to the University of Hradec Králové in November 2023, and he gratefully acknowledges the hospitality extended to him during this visit.

\end{document}